\documentclass[11pt, a4paper]{amsart}
\usepackage{amsmath, amssymb,amsthm}
\usepackage{mathrsfs}
\usepackage[all]{xy}
\usepackage{enumerate,enumitem, color, hyperref}

\newcommand{\A}{\mathbb{A}}
\newcommand{\C}{\mathbb{C}}
\newcommand{\R}{\mathbb{R}}
\newcommand{\Q}{\mathbb{Q}}
\newcommand{\Z}{\mathbb{Z}}

\DeclareMathOperator{\diag}{diag}
\DeclareMathOperator{\GL}{GL}
\DeclareMathOperator{\SL}{SL}
\DeclareMathOperator{\PGL}{PGL}
\DeclareMathOperator{\N}{N}
\DeclareMathOperator{\SO}{SO}

\DeclareMathOperator{\Hom}{Hom}

\DeclareMathOperator{\tr}{tr}

\DeclareMathOperator{\Res}{Res}
\DeclareMathOperator{\Ad}{Ad}
\DeclareMathOperator{\id}{id}
\DeclareMathOperator{\Sym}{Sym}
\DeclareMathOperator{\Crit}{Crit}
\DeclareMathOperator{\Ind}{Ind}
\DeclareMathOperator{\triv}{triv}
\DeclareMathOperator{\cyc}{cyc}
\DeclareMathOperator{\pr}{pr}

\newcommand{\gp}{\mathfrak{p}}
\newcommand{\gm}{\mathfrak{m}}

\newcommand{\lieg}{{\mathfrak{g}}}
\newcommand{\lieh}{{\mathfrak{h}}}
\newcommand{\liel}{{\mathfrak{l}}}
\newcommand{\liek}{{\mathfrak{k}}}

\newcommand{\lieb}{{\mathfrak{b}}}
\newcommand{\liet}{{\mathfrak{t}}}
\newcommand{\lien}{{\mathfrak{n}}}
\newcommand{\lieu}{{\mathfrak{u}}}

\newcommand{\lieq}{{\mathfrak{q}}}

\newcommand{\liea}{{\mathfrak{a}}}

\newcommand{\rH}{\operatorname{H}}

\newcommand{\cS}{\mathcal{S}}
\newcommand{\cO}{\mathcal{O}}
\newcommand{\cP}{\mathcal{P}}

\newcommand{\cG}{\mathcal{G}}
\newcommand{\cU}{\mathcal{U}}
\newcommand{\cE}{\mathcal{E}}

\newcommand{\cV}{\mathcal{V}}

\newcommand{\Cl}{\mathscr{C}\!\ell_F^+}
\newcommand{\SGK}{S^G_K} 
\newcommand{\SHbeta}{\tilde{S}^H_{L_\beta}}
\newcommand{\w}{{\sf w}}

\usepackage{bm}

\newcommand{\mmu}{\boldsymbol{\mu}}

\theoremstyle{plain}
\newtheorem{theorem}{Theorem}[section]
\newtheorem{lemma}[theorem]{Lemma}
\newtheorem{corollary}[theorem]{Corollary}
\newtheorem{prop}[theorem]{Proposition}

\newtheorem{definition}[theorem]{Definition}

\newtheorem{theoremletter}{Theorem}

\title[$p$-adic $L$-functions and non-vanishing]{$L$-functions of $\GL_{2n}:$ \\ $p$-adic properties and non-vanishing of twists }

\author{\bf Mladen Dimitrov,  \ \ Fabian Januszewski,  \ \ A.~Raghuram}

\address{University of Lille, CNRS, UMR 8524,  Laboratoire Paul Painlev\'e, 59000 Lille, France}
\email{mladen.dimitrov@univ-lille.fr,  mladen.dimitrov@gmail.com}

\address{Institut f\"ur Mathematik, EIM, Paderborn University, Germany}
\email{fabian.januszewski@math.upb.de}

\address{Indian Institute of Science Education and Research, Dr.\ Homi Bhabha Road, Pashan, Pune 411021, India}
\email{raghuram@iiserpune.ac.in, raghuram.1@gmail.com}

\subjclass[2010]{Primary: 11F67; Secondary: 11S40, 11F55, 11F70}

\begin{document}

\begin{abstract}
The principal aim of this article is to attach and study  $p$-adic $L$-functions to cohomological cuspidal automorphic representations $\Pi$ of $\GL_{2n}$ over a totally real field  $F$ admitting a Shalika model.  We use a modular symbol approach, along the global lines of the work of Ash and Ginzburg, but our results are more definitive since we draw heavily upon the methods used in the recent and separate works of all the three authors.  By construction our $p$-adic $L$-functions are distributions on the Galois group of the maximal abelian extension of $F$ unramified outside $p\infty$. Moreover,  we work  under a weaker  Panchishkine type condition  on $\Pi_p$
 rather than the full ordinariness condition. Finally, we prove the 
so-called Manin relations between the $p$-adic $L$-functions at {\it all} critical points. This has the striking consequence that, given a unitary $\Pi$ whose standard $L$-function admits at least two critical points, and given a prime $p$ such that $\Pi_p$ is ordinary, the central critical value $L(\tfrac{1}{2}, \Pi\otimes\chi)$ is non-zero for all except  finitely many Dirichlet characters $\chi$ of $p$-power conductor. 
\end{abstract}

${  }$
\vspace{-1cm}
\maketitle

\addtocontents{toc}{\setcounter{tocdepth}{0}}
\section*{Introduction}

\addtocontents{toc}{\setcounter{tocdepth}{2}}

A crucial result in Shimura's work on the special values of $L$-functions of modular forms concerns the existence of a twisting character to ensure that a twisted $L$-value is non-zero at the center of symmetry (see \cite[Thm.  2]{shimura-mathann}). Since then it has been a very important problem in the analytic theory of automorphic $L$-functions to find  characters to render a twisted $L$-value non-zero. 
Rohrlich \cite{rohrlich} proved such a non-vanishing result in the context of cuspidal automorphic representations of $\GL_2$ over any number field. This was then generalized to $\GL_N$ over any number field by Barthel-Ramakrishnan \cite{barthel-ramakrishnan} and further refined by Luo \cite{luo}. However, neither \cite{barthel-ramakrishnan} nor \cite{luo} can prove this at the center of symmetry if $N \geqslant 4.$ (For us the functional equation will be normalized so that $s = \tfrac{1}{2}$ is the center of symmetry.) There have been other types of analytic machinery that have been brought to bear on this problem,  for example, see
Chinta-Friedberg-Hoffstein \cite{chinta-friedberg-hoffstein}. Even for simple situations involving $L$-functions of higher degree this problem is open. For example,   suppose $\pi $ is the unitary cuspidal 
automorphic representation associated to a primitive holomorphic cusp form for $\GL_2/\Q$, then it has been an open problem to find a  Dirichlet character $\chi$ so that the twisted symmetric cube $L$-function 
$L\left(\tfrac{1}{2}, (\Sym^3\pi) \otimes \chi\right)$ is non-zero at the center. In this article, we prove the following result.

\begin{theoremletter}
\label{thm:main-theorem}
Let $F$ be a totally real field and $\Sigma_\infty$  the set of all its real places. 
Let $\Pi$ be a unitary cuspidal automorphic representation of $\GL_{2n}/F$ admitting a Shalika model and such that $\Pi_\infty$ is cohomological with respect to a pure dominant integral weight $\mu$ such that 
\begin{equation}\label{eq:two-critical}
\mu_{\sigma,n} >  \mu_{\sigma, n+1}, \quad \text{for all} \  \sigma\in \Sigma_\infty. 
\end{equation}

Assume that for all primes $\gp$ above a given prime number $p$, $\Pi_\gp$ is unramified and $Q$-ordinary, where $Q$ is the parabolic of type $(n,n)$ of $\GL_{2n}/F$ (see \eqref{eqn:ordinary}). 

Then, for all but finitely many Dirichlet characters $\chi$  of   $p$-power conductor we have: 
$$ L\left(\tfrac{1}{2}, \Pi \otimes (\chi\circ \N_{F/\Q}) \right) \neq 0.$$
\end{theoremletter}

For notions and notations that are not defined in the introduction, the reader will have to consult the main body of the paper. 
A more general statement is proven in Theorem \ref{thm:non-vanishing}. 
Furthermore, we can prove a stronger non-vanishing result covering the nearly-ordinary case (see Corollary \ref{cor:stronger-non-vanishing})
as well as a simultaneous non-vanishing result (see Corollary \ref{cor:simutaenous}). For example, with a classical normalization of $L$-functions, it follows from our results that there are infinitely many Dirichlet characters $\chi$ such that 
\[L(6, \Delta \otimes \chi) \cdot L(17, \Sym^3(\Delta) \otimes \chi) \neq 0\] for the Ramanujan $\Delta$-function. 
Our methods are purely arithmetic and involve studying $p$-adic distributions on $\Cl(p^\infty)$, the Galois group of the maximal abelian extension of $F$ unramified outside $p\infty$, that are attached to suitable eigenclasses in the cohomology of $\GL_{2n}$. 

\medskip

Let's now describe our methods and results in greater detail. We begin with a purely cohomological situation, without any reference to automorphic forms or $L$-functions. Let $\cO_{F}$ be the ring of integers of  $F$ and $\mathfrak{d}$ its different. 
Take a pure dominant integral weight $\mu$ for $G = \Res_{\cO_{F}/\Z}\GL_{2n},$  and let $V^\mu_{E}$ be the algebraic irreducible representation of $G(E)$, for some `large enough' $p$-adic field $E$.    If $\cO$ is the ring of integers of $E$, then we also consider an 
$\cO$-lattice  $V^\mu_{\cO}$   stabilized by $G(\cO)$.
For any open compact subgroup $K$ of  $\GL_{2n}$ over the  finite adeles of $F$, let $\cV_{\cO}^\mu$ be the associated sheaf on the locally symmetric space $\SGK$ of $G$ with level structure $K$  and let's consider the compactly supported cohomology $\rH^q_c (\SGK, \cV_{\cO}^\mu)$ endowed with the usual  Hecke action. 
 Assume that for  $\gp$ dividing $p$,  $K_\gp$ is the parahoric subgroup   corresponding to the parabolic $Q$  and
 consider an eigenclass $\phi\in \rH^{t}_c(\SGK, \cV_{\cO}^\mu)$ having a non-zero  eigenvalue $\alpha_\gp$ for a particular Hecke operator $U_\gp$. 
 Here and throughout the paper  $t= |\Sigma_\infty|(n^2+n-1)$ denotes the top degree supporting cuspidal cohomology. The weight $\mu$ determines a contiguous string of integers $\Crit(\mu)$ which would correspond to the  set of critical points for an $L$-function. For each $j \in \Crit(\mu)$ we attach an $E$-valued distribution $\mmu_{\phi}^{j}$ on  $\Cl(p^\infty)$ and  show that  it is $\cO$-valued when 
$\phi$ is $Q$-ordinary, {\it i.e.}, it is a measure (see  diagram  \eqref{eqn:evaluation-map} to get a quick overview of the sheaf-theoretic maps that are involved in the construction). 
Most importantly we prove in Theorem \ref{thm:manin} a  Manin type relation, namely  for all $j,j' \in \Crit(\mu)$ we have  
     $$ \varepsilon_{\cyc}^{j'-j}(\mmu^{j}_\phi)  =       \mmu^{j'}_\phi,     $$
    where  $\varepsilon : \Cl(p^\infty)  \to  \Z_p^\times$ is  the  $p$-adic cyclotomic character and 
      $\varepsilon_{\cyc}$ is   the automorphism of
    $\cO[[\Cl(p^\infty)]]$ sending  $[x]$ to $\varepsilon ([x])[x]$, allowing us  
to define a measure  $\mmu_\phi= \varepsilon_{\cyc}^{-j}(\mmu^{j}_\phi)$ which is independent of $j$.

\medskip

Next  we apply the above considerations  to the situation when $\phi$ is related to a cuspidal automorphic representation $\Pi$ of 
$\GL_{2n}/F$ such that $\Pi_\infty$ is cohomological with respect to the weight $\mu$ (see \S\ref{sect:cuspidal-coh}). 
Friedberg and Jacquet related the period integral of cusp forms in $\Pi$ over the   Levi subgroup $H$ of $Q$ to the standard $L$-function $L(s, \Pi)$, and 
for the unfolding of this integral to see the Eulerian property the representation is assumed to have a Shalika model (see \S\ref{sec:shalika}). 
Such a cohomological interpretation  was used in \cite{GR-ajm} to deduce algebraicity results for the critical values of 
$L(s, \Pi \otimes \chi)$. The following result further investigates their  $p$-adic integrality properties. 
A more general $p$-adic interpolation statement is proven in Theorem \ref{thm:p-adic-interpolation}  under the   
assumption that $\Pi_\gp$ admits a $Q$-regular refinement $\widetilde\Pi_\gp$  for $\gp\mid p$ (see Definition \ref{Q-regular}) which is shown to be  always fulfilled  when  $\Pi_\gp$ is  $Q$-ordinary (see Lemma \ref{lem:ord-reg}).

\begin{theoremletter} \label{thm:padic-L}
Let $\Pi$ be a cuspidal automorphic representation of $\GL_{2n}/F$ admitting a $(\psi,\eta)$-Shalika model and such that 
$\Pi_\infty$ is cohomological of weight $\mu$. Assume that for all primes $\gp$ above a given prime number $p$, $\Pi_\gp$ is spherical and   $Q$-ordinary, and let $\alpha_{\gp}$ denote the corresponding $U_{\gp}$-eigenvalue.  
Given any isomorphism $i_p: \C\xrightarrow{\sim} \bar\Q_p$,  there exists a bounded $p$-adic distribution $\mmu_{\widetilde\Pi}$ on $\Cl(p^\infty)$
 such that  for any $j\in \Crit(\mu)$ and for any finite order character $\chi$ of $\Cl(p^\infty)$ of  conductor $\beta_\gp \geqslant 1$ at  all $\gp\mid p$  one has
  \begin{multline*}
  i_p^{-1}\left(  \int_{\Cl(p^\infty)}\varepsilon^j(x)  \chi(x)  d\mmu_{\widetilde\Pi}(x)\right) 
    \ = \\ 
 =\gamma \cdot \N^{jn}_{F/\Q}(\mathfrak{d})   \cdot   
   \prod_{\gp\mid p}\left(\alpha_{\gp}^{-1} 
    q_\gp^{n(j+1)}\right)^{\beta_\gp} \cdot 
    { \cG(\chi_f)^{n} \cdot  L(j+\tfrac{1}{2}, \Pi_f \otimes \chi_f) }
     \zeta_\infty(j+\tfrac{1}{2};W_{\Pi_\infty,j}^{(\varepsilon^j\chi\eta)_\infty}), \text{ where}
  \end{multline*}
   $\cG(\chi_f)$  is the  Gauss sum, the zeta factor is non-zero by \eqref{eqn:sun} and $\gamma\in\Q^\times$ is as in \eqref{eq:gamma}.
\end{theoremletter}
 
 Let us hint on how we deduce Theorem \ref{thm:main-theorem}. Theorem \ref{thm:padic-L} whose formulation 
 implicitly uses the earlier established Manin relations gives  congruence relations between  successive critical values, 
 while \eqref{eq:two-critical}   translates   into $\tfrac{3}{2} \in \Crit(\mu)$. 
Since the complex $L$-function of the unitary cuspidal automorphic  representation 
$\Pi$ does not vanish for $\Re(s) \geqslant 1$ we deduce that 
$L(\tfrac{3}{2}, \Pi \otimes \chi)$ never vanishes, which in turn implies the
non-vanishing of $L(\tfrac{1}{2}, \Pi \otimes (\chi\circ \N_{F/\Q}))$
 for all but finitely many Dirichlet characters $\chi$ (see the proof of Theorem \ref{thm:non-vanishing}).

Let's mention some relevant papers in the literature. First of all, Ash and Ginzburg \cite{ash-ginzburg} started the study of 
$p$-adic $L$-functions for $\GL_{2n}$ over a totally real field by considering the analytic theory developed by Friedberg and Jacquet \cite{friedjac}. However
to quote the authors of \cite{ash-ginzburg}, their results are definitive only for $\GL_4$ over $\Q$ and for cohomology with constant coefficients. Furthermore, they construct their distributions on local units while only suggesting that one should really work, as we do in this paper, on $\Cl(p^\infty)$. This article uses the more recent  techniques developed in independent papers by all the three authors; namely, 
\cite{dim-padicL}, \cite{GR-ajm}, and \cite{J-imrn}. Finally, we mention Gehrmann's thesis \cite{gehrmann} which also constructs $p$-adic $L$-functions in essentially a similar context, but his methods are entirely different from ours.

To conclude the introduction, our emphasis is on the purely sheaf-theoretic nature of the construction of the distributions attached to eigenclasses in cohomology which leads to a purely algebraic proof of Manin relations in a very general context. When specialized to a cohomology class related to a representation $\Pi$ of $\GL_{2n}$, we get $p$-adic interpolation of the critical values of the standard $L$-function $L(s, \Pi)$,  and Manin relations give non-vanishing of twists $L(s, \Pi \otimes \chi)$ at the center of symmetry. A non-vanishing theorem in the realms of analytic number theory admitting 
a decidedly algebraic proof is philosophically piquant. 

\bigskip

\noindent {\it Acknolwedgements:} {\Small This project started when the three of us met at a conference in July 2014 at IISER Pune on $p$-adic aspects of modular forms. Any subset of two of the authors is grateful to the host institute or university of the third author during various stages of this work. MD and AR are grateful to an Indo-French research grant from CEFIPRA that has facilitated visits by each to the work-place of the other. MD acknowledges support from the ANR grants  CEMPI  (ANR-11-LABX-0007-01) and GALF  (ANR-18-CE40-0029).
AR is grateful to Charles Simonyi Endowment that funded his stay at the Institute for Advanced Study, Princeton.}

\newpage
\tableofcontents
\section{\bf Automorphic cohomology}
\label{sec:cohomology}

 Recall that $F$ is a totally real number field with  ring of integers $\cO_{F}$ and  set of infinite places $\Sigma_\infty$. 
For  a  set of places $\Sigma$, we denote by $\A^{(\Sigma)}$ the topological ring of adeles of $\Q$ outside $\Sigma$. Let 
$ \A_F=\A\otimes_{\Q} F$ (resp.  $\A_{F,f}$) be the group of adeles (resp. finite adeles) of $F$.

We consider  $G=\Res_{\cO_{F}/\Z}(\GL_{2n})$ as a reductive group scheme over $\Z$, quasi-split over $\Q$ and  let $Z=\Res_{\cO_{F}/\Z}(\GL_1)$ be the center of $G$. The standard Borel subgroup $B\subseteq  G$  is defined as the restriction of scalars of the standard Borel subgroup of all upper triangular matrices in $\GL_{2n}/\cO_{F}$. 
We have $B=TN$, where  $N$ is the unipotent radical of $B$ and $T$ is the standard torus of all diagonal matrices.  Let 
$H =\Res_{\cO_{F}/\Z}(\GL_n \times \GL_n)$, and $\iota : H \hookrightarrow G$ be the map that sends $(h_1,h_2)$ to 
$\left(\begin{smallmatrix} h_1 & 0 \\ 0 & h_2 \end{smallmatrix}\right)$.  Let $Q=HU$ be the standard parabolic subgroup of type $(n,n)$
whose  Levi subgroup  is $H$ and  unipotent radical is $U$. 
Finally, the Shalika subgroup $S$ of $G$ is defined as
$ S=\left\{   \left( \begin{smallmatrix} h &  0 \\ 0 &  h \end{smallmatrix}\right) 
\left( \begin{smallmatrix} 1 &  X\\ 0 &  1 \end{smallmatrix}\right) |
h\in \GL_n,  X\in {\rm M}_n \right\}$.

For any commutative ring $A$, we let $\lieg_A$, $\lieb_A$, $\lieq_A$, $\liet_A$, $\lieh_A$, $\lien_A$ and $\lieu_A$ stand for the Lie algebras of $G$, $B$, $Q$, $T$, $H$, $N$ and $U$ over $A$, respectively. For $\liea_A$ any amongst  these, we let $\cU(\liea_A)$ stand for the enveloping algebra over $A$.
In the particular case $A=\R$, let $\lieg_\infty=\lieg_\R\otimes_\R\C$ denote the complexification and likewise for the other groups.
The reader is referred to \cite{jantzen} as a general reference for integral Lie algebras and their enveloping algebras.

For any real reductive Lie group $\mathcal{G}$ we let $\mathcal{G}^\circ$ denote the connected component of the identity.
Let $G_\infty = G(\R),$ and similarly $Z_{\infty} = Z(\R)$.

 \subsection{Pure weights}
 \label{sec:weights-repns}
We identify integral weights $\mu$ of $T$ with tuples of weights $\mu=(\mu_{\sigma})_{\sigma\in \Sigma_\infty}$ where  $\mu_{\sigma}=(\mu_{\sigma,1},\dots ,\mu_{\sigma,2n})\in \Z^{2n}$.  A weight $\mu$ is $B$-dominant if
\begin{equation}\label{eq:B-dominant}
\mu_{\sigma,1}\geqslant \dots \geqslant \mu_{\sigma,2n}, \ \text{for all} \ \sigma\in \Sigma_\infty.
  \end{equation}
 Let $X^*_+(T)$ be the set of all such dominant integral weights.  For $\mu \in X^*_+(T)$ denote by  $V^\mu$ the unique algebraic irreducible rational representation of $G$ of highest weight $\mu$. 
 For any field  $E$ over which $\mu$ is defined, we denote by  $V^\mu_{E}$  its $E$-valued points. 
 Denote by  $\mu^\vee$  the highest weight of the contragredient $(V^{\mu})^{\vee}$ of $V^\mu$ which we consider as a rational character of $B$.

We call $\mu$ {\em pure} if there exists  $\w\in \Z$, called the {\em purity weight} of $\mu$, such that 
$$  V^\mu=V^{\mu^\vee}\otimes (\N_{F/\Q}\circ\det)^\w,$$ 
where $\N_{F/\Q}:\Res_{F/\Q}(\GL_1)\to \GL_1$ denotes the norm homomorphism. If $\mu$ is pure then 
\begin{equation}\label{eq:pure}
\mu_{\sigma, i} + \mu_{\sigma,2n-i+1} = \w, \text{  for all } \sigma\in \Sigma_\infty
\text{   and for all  } 1 \leqslant i \leqslant n.
  \end{equation}
  In particular, $\sum_{i=1}^{2n} \mu_{\sigma,i}=\w n$ is independent of $\sigma$. We let $X^*_0(T)\subset X^*_+(T)$ stand for the pure dominant integral weights of $T$. Given any $\mu \in X^*_0(T)$, define the set 
  \begin{equation}\label{eq:crit}
    \Crit(\mu)  =  \{j \in \Z \mid \mu_{\sigma,n} \ \geqslant \ j  \ \geqslant \ \mu_{\sigma,n+1}, \ \forall \ \sigma\in \Sigma_\infty \}.
  \end{equation}
   It is  well known that only pure weights support cuspidal cohomology, and 
  the motivation for this definition comes from the fact proved in \cite[Prop. 6.1]{GR-ajm} that 
  if $\Pi$ is a cuspidal automorphic representation of $G(\A)$ which is cohomological with respect to $\mu$ (see \S\ref{sect:cuspidal-coh}) then $\tfrac{1}{2}+j$ with $j \in \Z$ is critical for the standard $L$-function $L(s, \Pi \otimes \chi)$ for any finite order character $\chi$ if and only if $j \in \Crit(\mu)$. Note that  the central  point $\tfrac{\w+1}{2}$ of $L(s, \Pi \otimes \chi)$ is critical, 
({\it i.e.},     $\tfrac{\w}{2}\in \Crit(\mu)$)  if and only if $\w$ is even.

\subsection{Integral lattices}\label{sec:int-lattices}

  Let $E$ be a finite extension of $\Q_p$  and let $\cO$ be its ring of integers. 
  Given $\mu \in X^*_+(T)$, we consider $V^\mu_{E}$ as a representation of $G(E)$.

  Let $v_0 \in V^\mu_{E}$ be a non-zero lowest weight vector.  Then the unipotent radical $N^-(E)$ of the Borel subgroup $B^-(E)$ of lower triangular matrices fixes $v_0$, while  $T(E)$  acts on  $v_0$ via the character $-\mu^{\sf v} = w_{2n}(\mu)$ where 
$w_{2n}$ is the Weyl group element of longest length.

Observe that
  \begin{equation}
  \label{eqn:O-lattice-in-V-mu}
    V^\mu_{\cO}  =  \cU(\lien_\cO) v_0 
  \end{equation}
  is an $\cO$-lattice $V^\mu_{E}$ endowed with a natural action of  $G(\cO)$.

  We fix once and for all uniformizers $\varpi_{\gp}\in F_{\gp}$ and put 
  $t_{\gp} = \iota(\varpi_{\gp}\cdot{\bf1}_n,{\bf1}_n) \in \GL_{2n}(F_{\gp})$.
  Define  for any integral  multi-exponent $\beta=(\beta_{\gp})_{\gp\mid p}$ the element
  \begin{equation}
  \label{eqn:multi-exponent-tp}
    t_p^\beta=\prod_{\gp\mid p}t_{\gp}^{\beta_{\gp}}\in T(\Q_p),
  \end{equation}
and  consider the semi-group 
  \begin{equation}
  \label{eqn:delta-p-plus}
    \Delta_p^+=\{t_p^\beta\;\mid\;
 \;\beta_{\gp}\in \Z_{\geqslant 0}, \ \forall\gp\mid p\}.
  \end{equation}
  Then by our choice of dominance condition, we have for any $t\in\Delta_p^+$:
  \begin{equation}
  \label{eqn:adjoint-t-action}
    {\Ad}(t)Q(\cO)=tQ(\cO)t^{-1}\subseteq Q(\cO) \text{ and }  
    {\Ad}(t^{-1})U^-(\cO)=t^{-1}U^-(\cO) t\subseteq U^-(\cO).
  \end{equation}
  Consider the standard maximal  parahoric subgroup $J_p=\prod_{\gp\mid p} J_\gp\subset G(\Z_p)$, where 
   \begin{equation}
  \label{eqn:full-parahoric}
   J_\gp=t_\gp^{-1}\GL_{2n}(\cO_{F,\gp})  t_\gp\cap \GL_{2n}(\cO_{F,\gp}).
    \end{equation} 
    Since $J_p \supset Q(\Z_p)$ the parahoric decomposition is given by 
      \begin{equation}
  \label{eqn:parahoric-decomposition}
   J_p= (J_p\cap U^-(\Z_p))Q(\Z_p)=Q(\Z_p)(J_p \cap U^-(\Z_p)). 
  \end{equation}
  Using  \eqref{eqn:adjoint-t-action} and   \eqref{eqn:parahoric-decomposition} one sees that 
    \begin{equation}   \label{eqn:semi-group}
    \Lambda_p=J_p \Delta_p^+ J_p= Q(\Z_p) \Delta_p^+ (J_p\cap U^-(\Z_p))
      \end{equation}
 is a semi-group. Moreover since $U^-(\Z_p)\subset N^-(\Z_p)$ acts trivially on $v_0$, the
     $J_p$-action on $V^\mu_{\cO}$   extends uniquely  to an action $\bullet$ of the semi-group $\Lambda_p$
    by letting $\Delta_p^+$  act trivially on the lowest weight vector $v_0$. Then  for all $t\in\Delta_p^+$ and $v\in V^\mu_{\cO}$
    one has:
    \begin{equation}  \label{eqn:two-actions}
t\bullet v= \mu^\vee(t) (t\cdot v)
  \end{equation}
   In fact by  \eqref{eqn:O-lattice-in-V-mu} one can 
    write $v=m\cdot v_0$ for some $m\in \cU(\lien_\cO)$ and using \eqref{eqn:adjoint-t-action} one finds:
      \[
    t\bullet v= t\bullet (m\bullet v_0)  = {\Ad}(t)(m)\bullet(t\bullet v_0)={\Ad}(t)(m)\cdot v_0=\mu^\vee(t) {\Ad}(t)(m)(t\cdot v_0)=\mu^\vee(t) (t\cdot v).
      \]

\subsection{Local systems on locally symmetric spaces for $\GL_{2n}$}
\label{sec:local-systems}
The standard maximal compact subgroup of $G_\infty$ will be denoted $C_{\infty}=\prod_{\sigma\in \Sigma_\infty}C_\sigma$, where 
 $C_{\sigma}\simeq {\rm O}_{2n}(\R) $. The determinant identifies the   group of connected components  $C_{\infty}/C_{\infty}^\circ$ with  
$F_{\infty}^\times/F_{\infty}^{\times\circ}\cong \{\pm 1\}^{\Sigma_\infty}$. Let $K_\infty=C_{\infty} Z_{\infty}$ and for 
 any open compact subgroup $K$ of $G(\A_f)$ consider the  locally symmetric space:
\begin{equation}
\label{eqn:YK}
\SGK= G(\Q)\backslash G(\A)/K K_{\infty}^\circ  = G(\Q)\backslash \left((G_\infty/ K_{\infty}^\circ ) \times G(\A_f)/K\right).
\end{equation}
Note that $K_{\infty}^\circ  = C_{\infty}^\circ Z_{\infty}^\circ=C_{\infty}^\circ Z_{\infty}$ since $2n$ is even. In general, $\SGK$ is only a real orbifold. 
In the sequel we assume that $K$ is {\it sufficiently small} in the sense that for all $g\in G(\A)$,
\begin{equation}\label{eq:neat}
G(\Q)\cap gKK_{\infty}^\circ  g^{-1}= Z(\Q)\cap KK_{\infty}^\circ ,
\end{equation}
which implies in particular that $\SGK$ is a real manifold.

Given a  left  $G(\Q)$-module $V$ one can define $\cV_K$  as the  sheaf of locally constant sections of the local system: 
$$G(\Q)\backslash \left(G(\A) \times V \right)/KK_{\infty}^\circ \to \SGK,$$
where $\gamma(g,v)k=(\gamma g k,\gamma\cdot v)$ for all $\gamma\in G(\Q)$, $k\in K K_{\infty}^\circ $. 
 Consider the canonical fibration $\pi: (G(\R)/ K_{\infty}^\circ ) \times G(\A_f)/K \to \SGK$
given by going modulo the left action of $G(\Q)$. Then for any open $\cU \subset \SGK$ one has 
the sections $\cV_K(\cU)$ over $\cU$ to be the set of all locally constant $s : \pi^{-1}(\cU) \to V$  
such that $s(\gamma\cdot x) = \gamma\cdot s(x)$ for all $\gamma \in G(\Q), x \in \pi^{-1}(\cU).$ 
We denote by $\cV_{K,E}^{\mu}$ the sheaf associated to  $V^\mu_{E}$. 
The sheaf $\cV_{K,E}^{\mu}$ is 
non-trivial if and only if  
\begin{equation}\label{center}
\mu(Z(\Q)\cap K K_{\infty}^\circ )=\{1\}.
\end{equation}
 Condition (\ref{center}) is always satisfied if $\mu$ is pure, 
since $\det(F^\times \cap K K_{\infty}^\circ )\subset \cO_{F}^\times\cap F_{\infty}^{\times\circ}$.

In order to attach a sheaf to $V^\mu_{\cO}$ we need a slightly different construction.  
Given a left $K$-module $V$ satisfying \eqref{center}  define $\cV_K$ instead as the  sheaf of locally constant sections
of:
$$G(\Q)\backslash (G(\A)\times V) /K K_{\infty}^\circ   \to  Y_K,$$ 
with left $G(\Q)$-action and right $K K_{\infty}^\circ    $-action given by 
$\gamma(g,v)k=(\gamma g k , k^{-1}\cdot v)$.
Since $K$  acts on $V^\mu_{\cO}$ throught its $p$-component $K_p\subset G(\Z_p)\subset G(\cO)$ we obtain a
sheaf $\cV_{\cO}^\mu$ on $\SGK$. 

When the actions of  $G(\Q)$ and  $K$ on  $V$ extend compatibly into a left action of  $G(\A)$,  the two resulting  local systems  are  isomorphic by $(g,v)\mapsto (g,g^{-1}\cdot v)$, justifying  the abuse of notation.

\subsection{Hecke operators}\label{sec:hecke}
  For any  open compact subgroups $K'\subseteq  K$  of $ G(\A_f)$ the natural map 
  $p_{K',K} :  S^G_{K'}\to \SGK$  induces an isomorphism of sheaves $ p_{K',K}^*\cV_K\overset{\sim}{\to} \cV_{K'}$.
  
  When the $K$-action on $V$ extends to an action of a semi-group containing $K$ and $\gamma$, then one can define  a  Hecke operator
  $[K\gamma K]$ as a composition of three maps:
  $$
    [K\gamma K]=\mathrm{Tr}(p_{\gamma K\gamma^{-1}\cap K,K})\circ [\gamma] \circ  p_{K\cap \gamma^{-1}K\gamma,K}^* :\quad \rH_c^q(\SGK,\cV_K) \to \rH_c^q(\SGK,\cV_K), 
  $$
where $p_{K\cap \gamma^{-1}K\gamma,K}^*$ is the pull-back,  $ \mathrm{Tr}(p_{\gamma K\gamma^{-1}\cap K,K})$ is the finite flat trace and 
$$ [\gamma]:\rH_c^q(S^G_{K\cap \gamma^{-1}K\gamma},\cV_{K\cap \gamma^{-1}K\gamma})\to \rH_c^q(S^G_{\gamma K\gamma^{-1}\cap K},\cV_{\gamma K\gamma^{-1}\cap K})$$ 
is induced by the morphism of local systems given by $(g,v)\mapsto (g\gamma^{-1}, \gamma\cdot v)$ in the case of a right $K$-action.
  
  When $K_p\subset J_p$, the above construction applies  to $V^\mu_{\cO}$ on which the semi-group $\Lambda_p$ acts 
  by the $\bullet$-action (see \eqref{eqn:semi-group}) yielding for each   $t\in\Delta_p^+$ a Hecke operator $[KtK]$ on  
  $\rH_c^q(\SGK,\cV_{K,\cO}^\mu)$. 
Note that while the natural inclusion $V^\mu_{\cO}\subseteq  V^\mu_{E}$ is $K_p$-equivariant, it is not $\Lambda_p$-equivariant (see \eqref{eqn:two-actions}). As a consequence the natural map 
$\rH_c^q(\SGK,\cV_{K,\cO}^\mu)\to \rH_c^q(\SGK,\cV_{K,E}^\mu)$
is equivariant for the  $\bullet$-action of  $[KtK]$ on  the source and the action of  the optimally integral Hecke operator 
$[KtK]^\circ=\mu^\vee(t)[KtK]$ on the target. To ensure  compatibility with extension of scalars we will also denote   $[KtK]^\circ$ the 
Hecke operator $[KtK]$ acting (via the  $\bullet$-action) on $\rH_c^q(\SGK,\cV_{K,\cO}^\mu)$. 

For any prime $\gp\mid p$ of $F$ the following Hecke operators will play an important role: 
 \begin{equation}\label{eq:Up}    
 U_{\gp}  =     [K t_\gp K] \text{ and  }  U_\gp^\circ=\mu^\vee(t_\gp) U_\gp.
  \end{equation}

For  $\beta=(\beta_\gp)_{\gp\mid p}$ with $\beta_\gp\in\Z_{\geqslant 0}$ we  let
 $U_{p^\beta}  =     [K t_{p^\beta} K] $ and $U_{p^\beta}^\circ=\mu^\vee(t_{p^\beta}) U_{p^\beta}.$

  Since the image of $\rH^{q}_c(\SGK, \cV_{\cO}^\mu)$ in $\rH^{q}_c(\SGK, \cV_{E}^\mu)$
  is a finitely generated  $\cO$-module, we  may assume that $E$ is  large enough so that all  $U_{\gp}^\circ$-eigenvalues  belong to $\cO$.

\section{\bf Distributions attached to cohomology classes for $\GL_{2n}$}
\label{sec:distributions}

Let $F_p=F\otimes_\Q\Q_p=\prod_{\gp \mid p}F_{\gp}$. 
For  a prime $\gp\mid p $ of $F$ we denote by   $I_\gp$ (resp., $J_\gp$)  the standard Iwahori (resp., parahoric) subgroup of 
$K_\gp^\circ=\GL_{2n}(\cO_{F,\gp})$ consisting of elements whose reduction modulo the  $\gp$ belongs to $B(\cO_{F}/{\gp})$ 
(resp., to $Q(\cO_{F}/{\gp})$).  

 We let $K=K^{(p)}\times \prod_{\gp\mid p}K_\gp$ be an open compact subgroup of $G(\A_f)$ such that: 
\begin{enumerate}[label=(K\arabic*), ref=(K\arabic*)]
\item \label{condition-K1} $K^{(p)}$ is the principal congruence subgroup of modulus  $\gm$, an ideal of $\cO_F$ which is 
relatively prime to $p$, and  $K^{(p)}G(\Z_p)$ satisfies \eqref{eq:neat},   
\smallskip
\item \label{condition-K2} $\left(\begin{smallmatrix}
T_n(\cO_{F,\gp})&M_n(\cO_{F,\gp}) \\{\bf 0}_n&T_n(\cO_{F,\gp})
\end{smallmatrix}\right)  \subseteq  K_\gp \subseteq  J_\gp$ for all $\gp \mid p$. 
\end{enumerate}

An  important role will be played by the  matrix   $\xi\in \GL_{2n}(\A_F)$, where 
\begin{equation}\label{eq:xi}
 \xi_{\gp}=
  \begin{pmatrix}{\bf1}_n&w_n\\{\bf0}_n&w_n\end{pmatrix}\in \GL_{2n}(\cO_{F,\gp}),  \text{ for all } \gp\mid p \text{ , and } 
  \xi_v={\bf1}_{2n},  \text{ for all } v\nmid p. 
  \end{equation}
where ${\bf1}_n$ and ${\bf0}_n$ are the $n \times n$ identity and zero matrices, respectively,  and $w_n$ is the longest 
length element in the  Weyl group of $\GL_n,$ whose $(i,j)$-entry is $\delta_{i, n-j+1}$. 

We have
$\xi_{\gp}^{-1} = 
\left(\begin{smallmatrix}
{\bf1}_n&-{\bf1}_n\\{\bf 0}_n&w_n
\end{smallmatrix}\right)$. 
Once and for all we record the identities
\begin{align}\label{eq:xi1}
  \xi_{\gp}^{-1}\cdot\begin{pmatrix}A&B\\C&D\end{pmatrix}\cdot\xi_{\gp}
  &=
  \begin{pmatrix}A-C & (A-D+B-C)w_n\\w_n C & w_n (C+D) w_n\end{pmatrix}, \text{  and }  \\
  \xi_{\gp}\cdot\begin{pmatrix}A&B\\C&D\end{pmatrix}\cdot\xi_{\gp}^{-1}
  &=
  \begin{pmatrix}A+w_nC & w_n D w_n-A+Bw_n-w_nC\\w_nC & w_n D w_n-w_nC\end{pmatrix}. \label{eq:xi2}
\end{align}

\subsection{Automorphic cycles} 
\label{sec:shl}
For any open-compact subgroup $L \subset H(\A_f)$ we consider  the locally symmetric space: 
\begin{equation}
\tilde{S}^H_L   =  H(\Q) \backslash H(\A) /  L L_\infty^\circ\text{ , where }  L_\infty=H_\infty \cap K_{\infty} . 
\end{equation}

Note that for each $\sigma\in \Sigma_\infty$ one has 
$
L_\sigma^\circ \simeq  \left(\begin{smallmatrix}\SO_{n}(\R) & 0 \\ 0 &\SO_{n}(\R) \end{smallmatrix}\right)\R^{\times\circ}$.
As in \eqref{eq:neat}, $\tilde{S}^H_L $ is a real manifold when $L$ is {\it sufficiently small} in the sense that for all $h\in H(\A)$,
\begin{equation}\label{eq:neat-bis}
H(\Q)\cap h LL_{\infty}^\circ  h^{-1}= Z(\Q)\cap  LL_{\infty}^\circ.
\end{equation}

Recall the notation  $t_{\gp}= \iota(\varpi_{\gp}\cdot{\bf1}_n,{\bf1}_n)$ where 
 $\varpi_{\gp}$ is an uniformizer at $\gp\mid p$. Recall also that for 
 $\beta=(\beta_{\gp})_{\gp\mid p}$ with  $\beta_\gp\in\Z_{\geqslant 0}$ we let   $p^\beta=\prod_{\gp\mid p}\varpi_{\gp}^{\beta_{\gp}}$
 and  $t_p^\beta  = \prod_{\gp\mid p} t_{\gp}^{\beta_{\gp}} \in G(\Q_p)$. 
 
 For any ideal $\gm$ of $\cO_F$, we denote by  $I(\gm)$  the open-compact subgroup  of $ \A_{F,f}^\times$ of modulus $\gm$, and  we consider the strict idele class group: 
$$
\Cl(\gm)= F^\times\backslash \A_F^\times/I(\gm) F_\infty^{\times\circ}.
$$

  We let $L_\beta = L^{(p)} \prod_{\gp\mid p}  L_\gp^{\beta_\gp}$  be an open compact subgroup of $H(\A_f)$ such that: 
\begin{enumerate}[label=(L\arabic*), ref=(L\arabic*)]
\item \label{condition-L1} $L^{(p)}=K^{(p)}\cap H$ is the principal congruence subgroup of modulus  $\gm$, and 
\smallskip
\item  \label{condition-L2}  $L_\gp^{\beta_\gp}=H(F_\gp)\cap K_\gp\cap\xi t_{\gp}^{\beta_\gp}K_\gp t_{\gp}^{-\beta_\gp}\xi^{-1}$
 for all $\gp \mid p$. 
\end{enumerate}

Note that conditions  \ref{condition-K1} and  \ref{condition-L1}  imply \eqref{eq:neat-bis}, in particular     $\SHbeta$ is a real manifold.

\begin{lemma} \label{lem:Lbeta}
$L_\gp^{\beta_\gp}$ consists of elements $(h_1 ,  h_2)\in \GL_n(\cO_{F,\gp})\times \GL_n(\cO_{F,\gp})$ such that  
   $$ \iota(h_1, h_2) \in K_\gp \cap \left( \begin{smallmatrix}{\bf 1}_n &  \\ & w_n \end{smallmatrix} \right)  K_\gp
  \left(  \begin{smallmatrix}{\bf 1}_n &  \\ & w_n \end{smallmatrix} \right),  
   \ \text{and} \ h_1h_2^{-1} \in 1+ \varpi_\gp^{\beta_\gp}M_n(\cO_{F,\gp}).   $$ 
    \end{lemma}
    
\begin{proof}
By \eqref{eq:xi1} for all $(h_1,h_2) \in H(F_\gp)\cap K_\gp=\GL_n(\cO_{F,\gp})\times \GL_n(\cO_{F,\gp})$
one has  
$$t_{\gp}^{-\beta_{\gp}} \xi^{-1}
\left(\begin{smallmatrix}h_1 & \\ & h_2 \end{smallmatrix}\right)
 \xi t_{\gp}^{\beta_{\gp}}=\left( \begin{smallmatrix}h_1 & \varpi^{-\beta_{\gp}}(h_1 - h_2)w_n \\ & w_n h_2 w_n \end{smallmatrix} \right). $$
Hence $h_1- h_2 \in \varpi_\gp^{\beta_\gp}M_n(\cO_{F,\gp})$, and as
$\left(\begin{smallmatrix}{\bf1}_n&M_n(\cO_{F,\gp}) \\{\bf 0}_n&{\bf1}_n
\end{smallmatrix}\right)  \subseteq  K_\gp$ we obtain $(h_1,w_n h_2 w_n)\in K_\gp$.\end{proof}

Lemma \ref{lem:Lbeta} implies that the map $\left(1+ \varpi_\gp^{\beta_\gp}M_n(\cO_{F,\gp})\right)\times\cO_{F,\gp}^\times
\to \det(L_\gp^{\beta_\gp})$ sending $(x,y)$ to $(xy,y)$ is an isomorphism. 
By the strong approximation theorem for $\SL_{n}(\A_F)$  the map 
$$(h_1,h_2) \mapsto \left( \frac{\det(h_1)}{\det(h_2)}, \det(h_2) \right)
$$
  identifies  the set of connected components of $\SHbeta$ with a product of two idele class  groups:
\begin{equation}
\label{eqn:pi-0-SHL}
\pi_0(\SHbeta) \overset{\sim}{\longrightarrow} \Cl(p^\beta \gm) \times \Cl(\gm).
\end{equation}
  It is
easy to see that the  fibre $\SHbeta[\delta]$ of $[\delta]\in \pi_0(\SHbeta)$ is connected  of dimension 
 \begin{equation}
\label{eqn:q0}
t= |\Sigma_\infty|(n^2+n-1).
\end{equation}
If we consider a cohomology class on $\SGK$ in degree $t$, and pull it back to $\SHbeta[\delta]$, then we end up with a top-degree class. The degree $t$ happens to be the top-most degree with non-vanishing cuspidal cohomology of $\SGK$. This {\it magical numerology} is at the heart of what ultimately permits us to give a cohomological interpretation to an integral representing an $L$-value 
(see \cite{GR-ajm}) and allows us to study it's $p$-adic properties.

\subsection{Evaluation maps} Fix  $\mu \in X^*_0(T)$. 

\subsubsection{\bf Automorphic symbols} 
\label{sec:iota-beta}
 By   \ref{condition-L2}  the map
  \begin{equation}
  \label{eqn:iota-beta}
    \iota_\beta : \SHbeta \ \rightarrow \ \SGK, \quad \quad [h] \ \mapsto \ [\iota(h)\xi t_p^\beta], 
  \end{equation}
  is well-defined. Since  $\iota_\beta$ is proper by a  well-known result of Borel and Prasad 
  (see, for example, \cite[Lem. 2.7]{ash})  one can consider the  pull-back: 
  \begin{equation}
  \label{eqn:iota-beta*}
    \iota_\beta^* : \rH^q_c(\SGK, \cV_{\cO}^\mu) \longrightarrow \rH^q_c(\SHbeta, \iota_\beta^*\cV_{\cO}^\mu). 
  \end{equation}

\subsubsection{\bf Twisting}
\label{sec:tau-beta}
By \ref{condition-L1}  the map $\iota : \SHbeta \to \SGK$, $[h]\mapsto [\iota(h)]$ is  well-defined and 
proper. Since $\xi t_p^\beta\in \Lambda_p$, using the $\bullet$-action  from \eqref{eqn:two-actions} one can consider the map 
$$H(\A)\times V_{\cO}^\mu\to H(\A)\times V_{\cO}^\mu, \,\,(h,v)\mapsto (h, (\xi t_p^\beta)\bullet v)$$
inducing a homomorphism of sheaves 
$\tau_\beta^\circ : \iota_\beta^*\cV_{\cO}^\mu \longrightarrow  \iota^*\cV_{\cO}^\mu$
hence a map in cohomology 
  \begin{equation}
  \label{eqn:normalized-tau-beta-on-cohomology}
    \tau_\beta^\circ : \rH^q_c(\SHbeta, \iota_\beta^*\cV_{\cO}^\mu) \longrightarrow \rH^q_c(\SHbeta, \iota^*\cV_{\cO}^\mu).
  \end{equation}
Similarly using the natural action of  $G(E)$ on $V_{E}^\mu$ instead of the $\bullet$-action one  defines a map
 \begin{equation}
  \label{eqn:tau-beta-on-cohomology}
    \tau_\beta : \rH^q_c(\SHbeta, \iota_\beta^*\cV_{E}^\mu) \longrightarrow \rH^q_c(\SHbeta, \iota^*\cV_{E}^\mu), 
  \end{equation}
  and  $ \tau_\beta=\mu^\vee(t_p^{-\beta})\tau_\beta^\circ$, since by  \eqref{eqn:two-actions} one has
$  (\xi t_p^\beta) \bullet v= \mu^\vee(t_p^\beta) (\xi t_p^\beta)\cdot v$ for all $v\in V_{E}^\mu$.

\subsubsection{\bf Critical maps}
\label{sec:kappa-j}

   For $j_1,j_2 \in \Z$ let $V^{(j_1,j_2)}$ be the $1$-dimensional $H$-representation
  $$(h_1,h_2) \ \mapsto \ 
    \N_{F/\Q}(\det(h_1)^{j_1} \det(h_2)^{j_2}).  
  $$
  
Let   $V^{(j_1,j_2)}_{\cO}$ be a free rank one $\cO$-module on which the above defined natural  $H(\Z_p)$-action is 
extended to a $H(\Q_p)$-action by letting $p\in\Q_p^\times$ act trivially. Note that this action 
is similar to the $\Lambda_p$-action   on $V^\mu_{\cO}$ defined  in \S\ref{sec:int-lattices}.

  It follows from \cite[Prop. 6.3]{GR-ajm} that  $j \in \Crit(\mu)$ (see \eqref{eq:crit}) if and only if
  \begin{equation}
  \label{eqn:hom-H-one-dim}
    \dim\left(\Hom_H(V^\mu, V^{(j,\w-j)}) \right) = 1. 
  \end{equation} 
  Fix a non-zero $\kappa_j \in \Hom_H(V^\mu, V^{(j,\w-j)})$ normalized so as to get an integral map: 
  $$
    \kappa_j : V^\mu_{\cO} \ \rightarrow \ V^{(j,\w-j)}_{\cO}.
  $$
  Denoting  $\cV^{(j,\w-j)}_{\cO}$ the sheaf on  $\SHbeta$ attached to  $V^{(j,\w-j)}_{\cO}$ by the construction described in \S\ref{sec:local-systems}, one obtains a homomorphism:   
  \begin{equation}
  \label{eqn:kappa-j*}
    \kappa_j  : \rH^q_c(\SHbeta, \iota^*\cV_{\cO}^\mu)  \longrightarrow \rH^q_c(\SHbeta, \cV^{(j,\w-j)}_{\cO}). 
  \end{equation}

  Putting \eqref{eqn:iota-beta*}, \eqref{eqn:normalized-tau-beta-on-cohomology} and \eqref{eqn:kappa-j*} together, for or each $j \in \Crit(\mu)$, we get a map:
  \begin{equation}
  \label{eq:down-to-H}
    \kappa_j \circ \tau_\beta^\circ\circ \iota_\beta^* : \rH^q_c(\SGK, \cV_{\cO}^\mu) \longrightarrow \rH^q_c(\SHbeta, \cV^{(j,\w-j)}_{\cO}). 
  \end{equation}

\subsubsection{\bf Trivializations}
\label{sec:triv}
Given any $\delta\in H(\A_f)$ the map
$$\triv_\delta:H(\Q)\delta L_\beta H_\infty^\circ\times V^{(j,\w-j)}_{\cO}\to H(\Q)\delta L_\beta H_\infty^\circ\times  V^{(j,\w-j)}_{\cO}\, , \,
(\gamma\delta l h_\infty, v)\mapsto (\gamma\delta l h_\infty, l_p^{-1}\cdot v)
$$
is well-defined since  $H(\Q)\cap L_\beta H_\infty^\circ\subset \ker(\N_{F/\Q}\circ\det)$ acts trivially on $V^{(j,\w-j)}_{\cO}$. 
An easy check shows that $\triv_\delta$ induces a homomorphism of local systems 
$$\SHbeta[\delta] \times V^{(j,\w-j)}_{\cO}\to \left(\cV^{(j,\w-j)}_{\cO}\right)|_{\SHbeta[\delta]}$$
where $[\delta]$ denotes the image of $\delta$ in $\pi_0(\SHbeta)$, hence yields a homomorphism: 
$$\triv_\delta^*: \rH^q_c(\SHbeta[\delta], \cV^{(j,\w-j)}_{\cO}) \to 
\rH^q_c(\SHbeta[\delta], \Z)\otimes V^{(j,\w-j)}_{\cO}.$$

We will now render the trivializations independent of the  choice of $\delta\in[\delta]\in\pi_0(\SHbeta)$.
By definition, for any $\delta'\in H(\Q) \delta l H_\infty$ one has 
\begin{equation}\label{eq:dependence-delta}
\triv_{\delta'}^*= (\id\otimes l_p^{-1})\cdot \triv_{\delta}^*=\N_{F_p/\Q_p}^{-1} \left(\det(l_{1,p})^j\det(l_{2,p})^{\w-j}\right)\triv_{\delta}^*. 
\end{equation}
The $p$-adic cyclotomic character $\varepsilon $ seen as  idele class character 
$F^{\times} \backslash \A_F^{\times} \to \Z_p^\times$ sends $y$ to $ \N_{F_p/\Q_p}(y_p)|y_f|_F\prod\limits_{\sigma\in \Sigma_\infty} 
{\rm sgn}(y_\sigma)$,  is trivial on 
$F_\infty^{\times\circ}$ and given by $\N_{F_p/\Q_p}$ on $(\cO_{F}\otimes\Z_p)^\times$. Hence
\begin{equation}\label{eq:triv-delta}
\triv_{[\delta]}^*= \varepsilon \!\left(\det(\delta_{1}^j\delta_{2}^{\w-j})\right)
\triv_\delta^*  : \rH^q_c(\SHbeta[\delta], \cV^{(j,\w-j)}_{\cO}) \to 
\rH^q_c(\SHbeta[\delta], \Z)\otimes V^{(j,\w-j)}_{\cO}
\end{equation}
is independent of the particular choice of $\delta\in[\delta]\in\pi_0(\SHbeta)$.

\subsubsection{\bf Connected components and fundamental classes}
\label{sec:degree-zero-class}
Recall that for each $[\delta]\in\pi_0(\SHbeta)$, $\SHbeta[\delta]$ is a $t$-dimensional connected orientable real manifold and that choosing an orientation
amounts to choosing a fundamental class, {\it i.e.}, a basis  $\theta_{[\delta]}$ of its Borel-Moore homology $H_{t}^{\mathrm{BM}}(\SHbeta[\delta])\simeq \Z$. We choose such orientations in a consistent manner  when $\beta$ and $[\delta]$ vary  as follows. First, we fix, once and for all, an ordered basis on the tangent space of the symmetric space $H_\infty^\circ/L_\infty^\circ$ yielding fundamental classes $\theta_{\beta}$ of the connected 
components of identity $\SHbeta[1]$, when $\beta$ varies. Then for each  $[\delta]\in\pi_0(\SHbeta)$ we consider the isomorphism 
$\SHbeta[1]\underset{\sim}{\overset{\cdot\delta}{\longrightarrow}} \SHbeta[\delta]$ 
and define $\theta_{[\delta]}=\delta_*\theta_{\beta} $, which is clearly independent of the particular choice of $\delta\in [\delta]$. 
Capping with  $\theta_{[\delta]}$ and fixing a  basis of $V^{(j,\w-j)}_{\cO}$ (later in \eqref{eq:basis-j-line} 
we will fix a particular basis in order to compare evaluations at different $j$'s)
yields an isomorphism:
$$\rH^{t}_c(\SHbeta[\delta], \Z)\otimes V^{(j,\w-j)}_{\cO} \overset{\sim}{\longrightarrow} V^{(j,\w-j)}_{\cO}\overset{\sim}{\longrightarrow} \cO.$$ 
Combining this with  \eqref{eq:down-to-H} and \eqref{eq:triv-delta} gives homomorphisms:  
\begin{align}
\cE^{j,\w}_{\beta,\delta}=&(\relbar\cap \theta_{[\delta]}) \circ\triv_\delta^*\circ\kappa_j \circ \tau_\beta^\circ\circ \iota_\beta^*:  \rH^{t}_c(\SGK, \cV_{\cO}^\mu)\to \cO, \label{eq:eval} \\ 
\cE^{j,\w}_{\beta, [\delta]}= \varepsilon \!\left(\det(\delta_{1}^j\delta_{2}^{\w-j})\right) \cdot \cE^{j,\w}_{\beta,\delta}=&
(\relbar\cap \theta_{[\delta]}) \circ \triv_{[\delta]}^* \circ\kappa_j \circ \tau_\beta^\circ\circ \iota_\beta^* :  \rH^{t}_c(\SGK, \cV_{\cO}^\mu)\to \cO. \nonumber
\end{align}

\subsubsection{\bf Summing over the second component}
Consider a finite order $\cO$-valued idele class  character $\eta_0$ of $F$ which is trivial on  $I(\gm)$, in particular  unramified at all places above $p$.    The character $\eta=\eta_0|\cdot|_F^{-\w}$ will later play a role when we discuss Shalika models for 
  automorphic   representations of $G$. The following map provides a section of \eqref{eqn:pi-0-SHL}: 
\begin{equation}
\label{eqn:dxy}
\delta(x,y)  =  (\diag(xy, 1,\dots, 1), \diag(y, 1,\dots ,1)) \in H.  
\end{equation}
When $(x,y)\in (\A_F^\times)^2$ runs over a set of representatives of $\Cl(p^\beta\gm)\times \Cl(\gm)$, 
$\SHbeta[\delta(x,y)]$ runs over the set of connected components of $\SHbeta.$ 
Define the level $\beta$ evaluation: 
 \begin{equation}
\label{eqn:summing-2nd-var}
\cE^{j,\eta}_{\beta}= \sum_{[\bar x] \in \Cl(p^\beta)} \cE^{j,\eta}_{\beta,[\bar x]}[\bar x] \text{, where }  
\cE^{j,\eta}_{\beta,[\bar x]}=\sum_{[y] \in \Cl(\gm)} \sum_{[x]}  \eta_0([y]) \cE^{j,\w}_{\beta,[\delta( x,y)]},
\end{equation}
where the last sum runs over all $[x] \in \Cl(p^\beta\gm)$ mapping to $[\bar x]$ under the natural projection. 

The following diagram recapitulates the steps in the construction of $\cE^{j,\eta}_{\beta}$: 
  \begin{equation}
  \label{eqn:evaluation-map}
    \xymatrix{
  \rH^{t}_c(\SGK, \cV_{\cO}^\mu) \ar[rrr]^{ \kappa_j \circ \tau_\beta^\circ\circ \iota_\beta^*} \ar[dd]_{\cE^{j,\eta}_{\beta}} &&&
        \rH^{t}_c(\SHbeta, \cV^{(j,\w-j)}_{\cO})
           \ar[dd]_{\underset{[\delta]\in\pi_0(\SHbeta) }{\sum} (\relbar\cap \theta_{[\delta]}) \circ \triv_{[\delta]}^*}
\\
&&& \\
\cO[\Cl(p^\beta)]  && \cO[\Cl(p^\beta\gm)\times \Cl(\gm)]  \ar[ll]_-{ [(x,y)]\mapsto  \eta_0([y])[\bar x]} &  \cO[\pi_0(\SHbeta)]  
 \ar@{=}[l] 
    }
  \end{equation}

\subsection{Distributions on $\Cl(p^\infty)$}
  
The object of this section is to relate  when  $\beta$ varies the evaluation maps $\cE^{j,\eta}_{\beta}$ whose definition is summarized in \eqref{eqn:evaluation-map}. 

\subsubsection{\bf The distributive property}
\label{sec:distributionrelation}
Fix a  $\beta=(\beta_{\gp})_{\gp\mid p}$ with $\beta_{\gp}\in \Z_{>0}$ for all $\gp\mid p$.

\begin{theorem}
\label{thm:distribution}
Given a   prime  $\gp\mid p$ we let  $p^{\beta'}=p^\beta\gp$ and consider the canonical projection
 $\pr_{\beta',\beta} : \Cl(p^{\beta'}) \to \Cl(p^\beta)$. For all $ [x] \in \Cl(p^\beta)$ we have 
 $\cE^{j,\eta}_{\beta}\circ U_\gp^\circ=\pr_{\beta',\beta}\circ \cE^{j,\eta}_{\beta'},$ {\it i.e.}, 
   $$    
   \cE^{j,\eta}_{\beta,[x]}\circ U_\gp^\circ
   \ = \ 
   \sum_{[x'] \, \in \, \pr_{\beta',\beta}^{-1}([x])} \cE^{j,\eta}_{\beta',[x']}. 
   $$
\end{theorem}

\begin{proof} Using  \eqref{eq:eval}, \eqref{eqn:dxy} and \eqref{eqn:summing-2nd-var} one has to show that 
for all $[x] \in \Cl(p^\beta\gm), [y]\in \Cl(\gm)$:
$$ \cE^{j,\w}_{\beta,\delta(x,y)}\circ U_\gp^\circ
   = \sum_{[x'] \in \pr_{\beta',\beta}^{-1}([x])}   \N_{F_p/\Q_p}^j(u_{x'})\cdot   \cE^{j,\w}_{\beta',\delta(x',y)}, \text{ where}  $$
 $u_{x'}\in I(p^\beta)$ is such that $x'\in F^\times x u_{x'} F_\infty^{\times\circ}$. 
We proceed as in the proof of \cite[Prop. 3.4]{BDJ}.
Pulling back the definition of the Hecke operator $U_\gp^\circ$ (see \S\ref{sec:hecke})  by the automorphic symbols
(see \S\ref{sec:iota-beta}) and the twisting operators (see \S\ref{sec:tau-beta})  yields a commutative diagram
(we use implicitly that $p_{K^0(\gp),K}$ and $\pr_{\beta',\beta}$ have the same degree as 
$L_{\beta}/L_{\beta'}\simeq M_n(\cO/\gp)$): 
  $$\hspace{-5mm}
  \xymatrix{ \rH_c^{t}(\SGK,\cV^\mu_K) \ar[r]^{\!\!\!\! p_{K_0(\gp),K}^*}  \ar[d]_{\iota_{\beta'}^*} &   \rH_c^{t}(S^G_{K_0(\gp)},\cV^\mu_{K_0(\gp)})  \ar[d]_{\iota_{\beta'}^*}\ar[r]^{[t_\gp]}    &\rH_c^{t}(S^G_{K^0(\gp)},\cV^\mu_{K^0(\gp)})  \ar[rr]^{\mathrm{Tr}(p_{K^0(\gp),K})} \ar[d]_{\iota_{\beta}^*}  & &  \rH_c^{t}(\SGK,\cV^\mu_K)\ar[d]_{\iota_{\beta}^*}
\\ \rH^{t}_c(\tilde  S^H_{L_{\beta'}}, \iota_{\beta'}^*\cV^\mu_K) \ar@{=}[r] \ar[d]_{\tau_{\beta'}^\circ}& \rH^{t}_c(\tilde  S^H_{L_{\beta'}}, \iota_{\beta'}^*\cV^\mu_{K_0(\gp)})
 \ar[r]^{[t_\gp]}  \ar[d]_{\tau_{\beta'}^\circ} & 
 \rH^{t}_c(\tilde S^H_{L_{\beta'}}, \iota_{\beta}^*\cV^\mu_{K^0(\gp)}) \ar[rr]^{\mathrm{Tr}\left(\pr_{\beta',\beta} \right)}\ar[d]_{\tau_{\beta}^\circ}
 && \rH^{t}_c(\SHbeta, \iota_{\beta}^*\cV^\mu_K) \ar[d]_{\tau_{\beta}^\circ}
 \\ \rH^{t}_c(\tilde  S^H_{L_{\beta'}}, \iota^*\cV^\mu_K) \ar@{=}[r]  & \rH^{t}_c(\tilde  S^H_{L_{\beta'}}, \iota^*\cV^\mu_{K_0(\gp)})
 \ar@{=}[r]  & 
 \rH^{t}_c(\tilde S^H_{L_{\beta'}}, \iota^*\cV^\mu_{K^0(\gp)}) \ar[rr]^{\mathrm{Tr}(\pr_{\beta',\beta} )}
 && \rH^{t}_c(\SHbeta, \iota^*\cV^\mu_K), }   $$
  where the upper  $[t_\gp]$ is induced by the morphism $(g,v)\mapsto (g\cdot t_\gp^{-1}, t_\gp\bullet v)$ of local systems, 
whereas the   lower $[t_\gp]$ is induced by the morphism  $(h,v)\mapsto (h, t_\gp\bullet v)$. Then
$$\xymatrix{ \rH^{t}_c(\tilde S^H_{L_{\beta'}}[\delta(x',y)],\cV^{(j,\w-j)}_{\cO} ) \ar[rr]^{\mathrm{Tr}\left(\pr_{\beta',\beta} \right)}
\ar[d]_{\triv^*_{\delta(x',y)}}
 && \rH^{t}_c(\SHbeta[\delta(x,y)], \cV^{(j,\w-j)}_{\cO}) \ar[d]_{\triv^*_{\delta(x,y)}}\\
 \rH^{t}_c(\SHbeta[\delta(x',y)], \Z)\otimes V^{(j,\w-j)}_{\cO} \ar[rr]^{\cdot  \N_{F_p/\Q_p}^j(u_{x'})}
 &&\rH^{t}_c(\SHbeta[\delta(x,y)], \Z)\otimes V^{(j,\w-j)}_{\cO},  } $$
is  another  commutative diagram by  \eqref{eq:dependence-delta}, hence the claim. 
\end{proof}

\subsubsection{\bf Distributions for finite slope eigenvectors}
 
  Let $\phi\in \rH^{t}_c(\SGK, \cV_{\cO}^\mu)$ be an eigenvector for   $U_{\gp}^\circ$ with 
  eigenvalue $\alpha_\gp^\circ$ for all  $\gp\mid p$.    Then, for  all $\beta=(\beta_\gp)_{\gp\mid p}$, it  is an eigenvector for  $U_{p^\beta}^\circ$  with eigenvalue
  $\alpha_{p^\beta}^\circ=\prod_{\gp\mid p}(\alpha_\gp^\circ)^{\beta_\gp}$. 
  We say that $\phi$ is of {\it finite slope}  if $\alpha_{p}^\circ \neq 0$ and in which case we define its slope as $v_p(\alpha_p^\circ)$. 
  A eigenvector $\phi$ of slope $0$  is called $Q$-ordinary. 
Being $Q$-ordinary  is equivalent to saying that the $U_{\gp}$-eigenvalue $\alpha_\gp$ satisfies  
$|\alpha_\gp|_p = |\mu^\vee(t_\gp)|_p^{-1}$ for all $\gp\mid p$ (see  \S\ref{sec:local-aspects-ordinary-zeta} for more details). 
 
 Given any $U_p$-eigenvector $\phi$ of finite slope and any $j\in  \Crit(\mu)$ by Theorem \ref{thm:distribution}   one has a well-defined element 
\begin{equation}
\label{eq:definitionofmu-fs}
    \mmu^{j,\eta}_\phi=\left((\alpha_{p^\beta}^\circ)^{-1} \cE^{j,\eta}_{\beta}(\phi)\right)_\beta
\end{equation}
which is thought of  as an $E$-valued distribution  on $\Cl(p^\infty)$.

       We write $\rH^{t}_c(\SGK, \cV_{\cO}^\mu)^{Q-\mathrm{ord}}$ for the maximal  $\cO$-submodule of $\rH^{t}_c(\SGK, \cV_{\cO}^\mu)$ on which the operators $U_{\gp}^\circ$ are invertible for all $\gp\mid p$ (it is a direct $\cO$-factor). Given any (not necessarily $U_p^\circ$-eigen) non-torsion element 
 $\phi\in \rH^{t}_c(\SGK, \cV_{\cO}^\mu)^{Q-\mathrm{ord}}$ one defines 
\begin{equation}
\label{eq:definitionofmu}
    \mmu^{j,\eta}_\phi=\left(\cE^{j,\eta}_{\beta}((U_{p^\beta}^\circ)^{-1}(\phi))\right)_\beta
 \in \cO[[\Cl(p^\infty)]]=\varprojlim_{\beta}\cO[\Cl(p^\beta)],  
\end{equation}
which can be reinterpreted as a measure ({\it i.e.}, a bounded distribution) on $\Cl(p^\infty)$.

\subsection{Manin relations}
\label{sec:manin}
    
  Consider the  $p$-adic cyclotomic character $\varepsilon : \Cl(p^\infty)  \to  \Z_p^\times$ which is defined by composing
    the norm  $\N_{F/\Q}:  \Cl(p^\infty)\to  \mathcal{C}\ell_\Q^+(p^\infty)$ with the  $p$-adic cyclotomic character 
    over $\Q$.     In this section we will prove the following result.

  \begin{theorem}
  \label{thm:manin}
    Let $\mu \in X^*_0(T)$ and suppose that $j$ and  $j+1$ both belong to   $\Crit(\mu)$. 

    For $\phi\in \rH^{t}_c(\SGK, \cV_{\cO}^\mu)^{Q-\mathrm{ord}}$ 
      the following  equality holds in $\cO[[\Cl(p^\infty)]]$:   
     $$ \varepsilon_{\cyc}(\mmu^{j,\eta}_\phi)  =       \mmu^{j+1,\eta}_\phi,     $$
    where  $\varepsilon_{\cyc}$ denotes  the automorphism of
    $\cO[[\Cl(p^\infty)]]$ sending  $[x]$ to $\varepsilon ([x])[x]$.  Hence 
     \begin{equation}\label{eq:def-mu-pi}
    \mmu_\phi^\eta=\varepsilon_{\cyc}^{-j}(\mmu^{j,\eta}_\phi)\in \cO[[\Cl(p^\infty)]],
  \end{equation}  
    is independent of $j\in \Crit(\mu)$. 
      \end{theorem}

The overdetermination of $\mmu^{\eta}_\phi$ in the $Q$-ordinary case, when there are at least two critical values,   
    will play a pivotal role in the proof of main theorem. Before embarking on the proof of this theorem, we begin with some technical preparation (see  \cite[\S3]{J-preprint-2}).

\subsubsection{\bf Lie theoretic considerations}

By the distributive property (see Theorem \ref{thm:distribution}) we may reduce to strict $p^\beta$-power level with integral exponents $\beta\in\Z_{>0},$ ignoring the finer components $\gp\mid p$ for simplicity of notation.  Recall  that $\lieb=\liet\oplus\lien$ and $  \lieq=\lieh\oplus\lieu$. With  the notation  $t_p=\iota(p{\bf1}_n,{\bf1}_n)$, 
we observe for any $\beta\geqslant 0$ the relations
$$
    t_p^\beta\lien_\cO t_p^{-\beta}\subseteq \lien_\cO, \quad 
    t_p^\beta\lieu_\cO t_p^{-\beta}=p^\beta\lieu_\cO.
$$
  Recall the matrix
  $ \xi = \left(\begin{smallmatrix}
    {\bf1}_n&w_n\\
    {\bf0}_n&w_n
    \end{smallmatrix}\right).
  $
  A superscript ${}^\xi(-)$ denotes left conjugation action by $\xi$.

\begin{prop}
\label{prop:transversality}
 We have the relations
 \begin{enumerate}
   \item  $   \lieg_\cO=\lieh_\cO+ {}^\xi\,\lieb_\cO^-$, and  
 \item    $ {}^\xi\!\left(\lien_\cO\cap\lieh_\cO\right)    \subseteq 
        [\lieh,\lieh]_\cO+        {}^\xi\,\lien_\cO^-$.
  \end{enumerate}
\end{prop}
\begin{proof} (i) 
  Since $\xi\in G(\cO)$, it suffices to verify it over $E$, where it amounts to show that $\dim_E\left(\lieh_E\cap{}^\xi\,\lieb_E^-\right) = n$.
To this end, let $l_1,l_2$ be lower triangular matrices in $M_n(E)$ and $u\in M_n(E)$. Then
\begin{eqnarray*}
\xi\cdot
\begin{pmatrix}
l_1&\\
u&l_2
\end{pmatrix}\cdot
\xi^{-1}
&=&
\begin{pmatrix}
  l_1+w_nu& w_n l_2 w_n-l_1-w_nu\\
  w_nu& w_n l_2 w_n-w_nu
\end{pmatrix}
\end{eqnarray*}
lies in $\lieh_E$ if and only if
$u = 0$ and $l_1 = w_n l_2 w_n$.
Therefore, $l_1$ and $l_2$ are diagonal matrices determining each other uniquely. 

(ii) Conjugation   by $\xi^{-1}$ reduces the claim  to the problem of solving
$$
\begin{pmatrix}
  n_1&\\&n_2
\end{pmatrix}=
\begin{pmatrix}
  h_1&(h_1-h_2)w_n\\&w_n h_2 w_n
\end{pmatrix}
\;+\;
\begin{pmatrix}
  \overline{n}_1&\\&\overline{n}_2
\end{pmatrix}
$$
for given $\iota(n_1,n_2)\in\lieh\cap\lien_\cO$ and unkowns $\iota(h_1,h_2)\in[\lieh_\cO,\lieh_\cO]$ and $\iota(\overline{n}_1,\overline{n}_2)\in\lien^-$.  The choice
$$
h_1=h_2=n_1+w_n n_2 w_n,\quad \overline{n}_1=-w_n n_2 w_n,\quad \overline{n}_2=-w_n n_1 w_n,
$$
is a solution with the desired properties.
\end{proof}

\begin{corollary}\label{cor:productrelations}
For any $\beta\geqslant0$, the following  relations  hold inside $\cU(\lieg_\cO)$:
\begin{enumerate}
   \item  $\cU(\lieg_\cO)= \cU(\lieh_\cO) \cdot \cU({}^\xi \lieb_\cO^-)$\text{, and } 
\item   $\cU({}^{\xi t_p^\beta}\lien_\cO)\subseteq 
  \cU([\lieh,\lieh]_\cO+p^\beta\lieh_\cO)\cdot
  \cU({}^\xi\lien_\cO^- + p^\beta \, {}^\xi\lieb_\cO^-)$. 
\end{enumerate}
\end{corollary}

\begin{proof}
  (i) This is a consequence of Proposition \ref{prop:transversality}(i) and the Poincar\'e--Birkhoff--Witt Theorem. 
  
  (ii) The decomposition $\lien_\cO = \left(\lieh_\cO\cap\lien_\cO\right)\oplus\lieu_\cO,$ gives
$t_p^\beta\lien_\cO t_p^{-\beta} = 
  \left(\lieh_\cO\cap\lien_\cO\right)\oplus p^\beta\lieu_\cO$. Conjugating by $\xi$ we get 
$$
  \xi t_p^\beta\lien_\cO t_p^{-\beta}\xi^{-1}=
 {}^\xi \! \left(\lieh_\cO \cap \lien_\cO \right)\oplus p^\beta \, {}^\xi\lieu_\cO. 
$$
Applying  Proposition \ref{prop:transversality}(ii) to the first summand and Proposition \ref{prop:transversality}(i) to the second we get
$$
  \xi t_p^\beta\lien_\cO t_p^{-\beta}\xi^{-1}\subseteq 
  \left([\lieh,\lieh]_\cO+p^\beta\lieh_\cO\right)\;+\;
  \left({}^\xi\lien_\cO^- + p^\beta \, {}^\xi \lieb_\cO^- \right).
$$
One concludes again by the Poincar\'e-Birkhoff-Witt Theorem, because the sums within the parentheses on the right hand side are Lie $\cO$-algebras.
\end{proof}

\subsubsection{\bf Lattices and the projection formula}

Recall from \eqref{eqn:O-lattice-in-V-mu} the lowest weight vector $v_0 \in V^\mu_{E}$ and the $G(\cO)$-lattice
$V^\mu_{\cO} = \cU(\lieg_\cO)\cdot v_0 = \cU(\lien_\cO)\cdot v_0$. 
Recall also the  $\bullet$-action of the semi-group $\Lambda_p$ on $V^\mu_{\cO}$ as in \eqref{eqn:two-actions}.

Given  $j \in \Crit(\mu)$ recall from \S\ref{sec:kappa-j} the map
$\kappa_j: V^\mu_{\cO} \to V^{(j,\w-j)}_{\cO}$. By Corollary \ref{cor:productrelations}(i)
$$
V^\mu_{\cO}=\cU(\lieh_\cO)\cdot \xi v_0,
$$
which implies that  $\kappa_j(\xi v_0)$ is an $\cO$-basis of  $V^{(j,\w-j)}_{\cO}$ yielding a surjective $\cO$-linear  map
\begin{equation}\label{eq:basis-j-line}
\kappa_j^\circ : V^\mu_{\cO} \to \cO, \quad \mbox{defined by} \quad 
\kappa_j(v) = \kappa_j^\circ(v) \kappa_j(\xi v_0). 
\end{equation}
It is independent from the choice of $\kappa_j$ because of \eqref{eqn:hom-H-one-dim},  and $\kappa_j^\circ(\xi v_0)=1$. We now come to the main technical result that is at the heart of our proof of the Manin relations.

\begin{prop}
\label{prop:manin-technical}
  For any $\beta\geqslant 0$, $v\in  (\xi t_p^\beta)\bullet V^\mu_{\cO}  \subset V^\mu_{\cO}$ and for all 
   $j,j'\in \Crit(\mu)$ we have
  \begin{equation}    \label{eq:manin-technical}
    \kappa^\circ_{j}(v)\;\equiv\;\kappa^\circ_{j'}(v)\pmod{p^\beta}.
  \end{equation}
\end{prop}

\begin{proof}
 By \eqref{eqn:two-actions}  for $v\in  (\xi t_p^\beta)\bullet V^\mu_{\cO} $ there exists $m\in \cU(\lien_\cO)$ with
$$
    v = \xi\cdot t_p^\beta \bullet (m v_0) = {}^{\xi t_p^\beta} m \cdot\xi (t_p^\beta \bullet v_0) 
    = {}^{\xi t_p^\beta} m \cdot\xi v_0 
    \ \in \ \cU({}^{\xi t_p^\beta}\lien_{\cO}) \cdot\xi v_0. 
$$
  By Corollary \ref{cor:productrelations}(ii)  write 
$$
    {}^{\xi t_p^\beta} m = xy\text{, with }   x \in \cU([\lieh,\lieh]_\cO+p^\beta\lieh) \text{  and }  
 y \in \cU({}^\xi\lien_\cO^-+p^\beta{}^\xi\lieb_\cO^-).   
 $$
Let $x_0, y_0  \in \cO$ be the degree zero terms of $x$ and $y$, respectively, and let 
  \begin{eqnarray*}
    x_1 = x - x_0 &\in&
    ([\lieh,\lieh]_\cO+p^\beta\lieh)\cdot \cU([\lieh,\lieh]_\cO+p^\beta\lieh),\\
    y_1 = y - y_0 &\in&
    ({}^\xi\lien_\cO^-+p^\beta{}^\xi\lieb_\cO^-)\cdot \cU({}^\xi\lien_\cO^-+p^\beta{}^\xi\lieb_\cO^-),
  \end{eqnarray*}
be the higher degree terms in their respective enveloping algebras. Then
  \begin{align*}
    \kappa_j^\circ(v)&= \kappa_j^\circ(xy \cdot \xi v_0) = 
    x \cdot \kappa_j^\circ(y \cdot \xi v_0) & \quad \mbox{(since $\kappa_j^\circ$ is $H$-equivariant)}\\
    &\equiv
    x \cdot \kappa_j^\circ(y_0 \cdot \xi v_0)\pmod{p^\beta}   & \quad \mbox{(since ${}^\xi\lien_\cO^-$ acts trivially on $\xi v_0$)}\\
     &\equiv
    x_0 \cdot\kappa_j^\circ(y_0\cdot\xi v_0)\pmod{p^\beta} & \quad \mbox{(since $[\lieh,\lieh]_\cO$ acts trivially on a line)}\\
&=
    x_0y_0\cdot\kappa_j^\circ(\xi v_0) = x_0y_0, &
  \end{align*}

  which does not depend on $j$ as claimed.
\end{proof}

\subsubsection{\bf Proof of Theorem~\ref{thm:manin}}
Since $\cO[[\Cl(p^\infty)]]= \varprojlim_{\beta}(\cO/p^\beta\cO)[\Cl(p^\beta)]$ and since $\varepsilon \pmod{p^\beta}$ factors through $\Cl(p^\beta)$, it is enough to check that given $\beta\geqslant 1$ and $[x]\in \Cl(p^\beta)$ one has 
$$ \varepsilon ([x]) \cE^{j,\eta}_{\beta,[x]}(\phi) \equiv \cE^{j+1,\eta}_{\beta,[x]}(\phi) \pmod{p^\beta}.$$ 
Since by  \eqref{eq:triv-delta} one has  $\triv_{[\delta(x,y)]}^*= \varepsilon (x^j y^{\w})\triv_{\delta(x,y)}^*$, 
it suffices to show that (see \eqref{eqn:evaluation-map}) 
$$ \cE^{j,\eta}_{\beta,\delta}(\phi)=(\relbar\cap \theta_{[\delta]}) \circ\triv_\delta^*\circ\kappa_j \circ \tau_\beta^\circ\circ \iota_\beta^*(\phi)\equiv
 (\relbar\cap \theta_{[\delta]}) \circ \triv_\delta^*\circ\kappa_{j+1}^* \circ \tau_\beta^\circ\circ \iota_\beta^*(\phi)\pmod{p^\beta}.$$ 
Now, by definition    the homomorphism of sheaves   $\tau_\beta^\circ $ defined in \S\ref{sec:tau-beta}  factors as:
  $$  \iota_\beta^*\cV_{\cO}^\mu \longrightarrow\iota^*\left((\xi t_p^\beta)\bullet\cV_{\cO}^\mu\right) \longrightarrow  \iota^*\cV_{\cO}^\mu. $$
  Hence   Proposition \ref{prop:manin-technical} translates to the statement that (for the choice of  basis of $V^{(j,\w-j)}_{\cO}$ 
   as in \eqref{eq:basis-j-line})  
   $(\relbar\cap \theta_{[\delta]}) \circ \triv_\delta^*\circ\kappa_j \circ \tau_\beta^\circ\circ \iota_\beta^*(\phi)\pmod{p^\beta}$ 
   is independent of $j\in\Crit(\mu)$.\hfill $\square$

\section{\bf Local considerations}
\label{sec:local}

We delineate some local calculations that will be needed in the  global considerations of the next section. 
For only this  section, $F$ denotes a finite extension of $\Q_p$, $\cO$ its ring of integers, $\mathcal{P}$  the maximal ideal, 
 $\varpi\in \mathcal{P}$ a uniformizer, $q=\#(\cO/\cP)$ and $\delta$  
the valuation of the different. 
 We use local notations  corresponding to the global notations introduced at the beginning of \S\ref{sec:cohomology}. For example
   $G=\GL_{2n}(F)\supset H=\GL_{n}(F)\times \GL_{n}(F)$, {\it etc.}

\subsection{Parahoric invariants}
\label{sec:parahoric}
Let $K= \GL_{2n}(\cO)$ be the standard maximal compact subgroup of $G$. 
Define  the  parahoric (resp., Iwahori) subgroup $J$ (resp., $I$) of  $K$ consisting of matrices whose reduction modulo $\mathcal{P}$ belongs to  $Q(\cO/\cP)$ (resp., to $B(\cO/\cP)$). One has:
\begin{equation}
\label{eqn:parahoric}
J  =  
\left\{  \begin{pmatrix} a & b \\ c & d \end{pmatrix} \in K \ \Big{|} \ 
a, d \in \GL_n(\cO),  c \in M_n(\mathcal{P}),  b \in M_n(\cO) \right\}.  
\end{equation}

 Let   $\Pi$ be an algebraic unramified and generic  representation of $G$. Then there exists an
unramified character $\lambda = \otimes_{i=1}^{2n} \lambda_i :T \to\bar\Q^\times\xrightarrow{i_\infty}\C^\times$ such that    
\begin{equation}\label{eq:spherical}
   \Pi=  \Ind_{B}^{G}(|\cdot|^{\frac{2n-1}{2}}\lambda),
  \end{equation}
where the right hand side is the normalized parabolic induction, which differs from the usual induction by $\delta_B^{\frac{1}{2}}$ where 
$$\delta_B(t_1,t_2,\dots,t_{2n})= |t_1 |^{2n-1} |t_2 |^{2n-3} \cdot\dots\cdot |t_{2n}|^{1-2n}.$$

Recall  Jacquet's exact  functor sending an admissible $G$-representation $V$ to the space of its 
co-invariants of $U$ defined as $V_U=V /\langle \{u\cdot v-v\mid u\in U, v\in V \} \rangle$ which  is an 
admissible $H$-representation. 
The Weyl groups of $G\supset H$ are   given by $ \mathfrak{S}_{2n} \simeq W_G \supset W_H\simeq (\mathfrak{S}_{n}\times\mathfrak{S}_{n}) $.  The group $W_G$ acts on the right on characters of $T$. There is a natural bijection: 
     \begin{equation}\label{eq:tau}
W_G / W_H\xrightarrow{\sim}  \{\tau\subset \{1, 2, \dots ,2n\} |\#\tau=n\},\,\ \rho\mapsto \{\rho(1), \dots ,\rho(n)\}
          \end{equation}

   \begin{lemma}  \label{lemma-Jacquet}
   The semi-simplification of  the Jacquet module  $\Pi_U$ is isomorphic to: 
   \begin{equation}\label{eq:Q-Jacquet}
   \bigoplus_{ \tau\in W_G / W_H} \delta_Q^{\frac{1}{2}}\cdot \Ind_{B\cap H}^H(|\cdot|^{\frac{2n-1}{2}}\lambda^\tau),
    \end{equation}  
     where 
$\delta_Q(t_1,t_2,\dots,t_{2n})= \left|{t_1\dots  t_n}\cdot{t_{n+1}^{-1}\dots  t_{2n}^{-1}} \right|^n$. The   semi-simplification  can be omitted if  
  $\Pi$ is regular in the  sense that $\alpha_{i} = \lambda_i(\varpi)$  are pairwise distinct for $1\leqslant i\leqslant 2n$. 
  
The characteristic polynomial  of  the Hecke operator $U_\cP=[J t_\varpi J ]$ acting on $\Pi^{J }$ equals      
\[ \prod_{ \tau\in W_G / W_H}\left(X-q^{\frac{n(1-n)}{2}} \prod_{i\in \tau} \alpha_{i}\right).\]          
  \end{lemma}

  \begin{proof}  The semi-simplification  of the Jacquet module $\Pi_N$   with respect to $B$ is given by: 
   \begin{equation}\label{eq:B-Jacquet}
   \bigoplus_{ \rho\in W_G} \delta_B^{\frac{1}{2}}|\cdot|^{\frac{2n-1}{2}}\lambda^\rho.
    \end{equation}  
Since $ \Ind_{B}^G=  \Ind_{Q}^G \Ind_{B\cap H}^H$, Frobenius reciprocity implies that any  
 irreducible sub-quotient of   the Jacquet module of $\Pi$ with respect to $Q$ is isomorphic to one of the summands in \eqref{eq:Q-Jacquet}. The first claim then follows by a simple dimension count based on \eqref{eq:B-Jacquet} and the transitivity of the Jacquet functors. 
 By Bruhat decomposition: 
    \begin{align}\label{eq:bruhat}
G=\coprod_{\rho \in W_G} B\rho I=\coprod_{\rho \in W_G/ W_H} B\rho J, \,\,\,
K=\coprod_{\rho \in W_G} (B\cap K) \rho I=\coprod_{\rho \in W_G/ W_H} (B\cap K)  \rho J, 
\end{align}
 the dimension of $\Pi^{J}$ is $\#(W_G/W_H)$. By  the Iwasawa decomposition $H= (B\cap H)\cdot (H\cap J)$,  
$$\left(\Ind_{B\cap H}^H(\delta_Q^{\frac{1}{2}}|\cdot|^{\frac{2n-1}{2}}\lambda^\tau)\right)^{H\cap J}$$
is a line  on which the central element   $\iota({\bf1}_{n},\varpi{\bf1}_{n})$ acts by $q^{\frac{n(1-n)}{2}} \prod_{i\in \tau} \alpha_{i}$. 
Under the  assumption that $\Pi$ is regular,  the image of $\Pi^{J }$ by  the Jacquet functor equals 
    the direct sum of the above lines when  $\tau$ runs over $W_G / W_H$, hence the  second claim. 
The proof of the  third claim is a standard double coset computation based on  \eqref{eq:bruhat} (see also \cite{Hida1998}). 
 \end{proof}  
\subsection{Twisted local Shalika integrals}
We will review the  theory of global  Shalika models and $L$-functions in \S\ref{sec:shalika}. 
The computations in this section will be needed in \S\ref{sec:p-adic-interpol} to  evaluate the  twisted local zeta integral.

Fix an additive character $\psi:F\to\C^\times$ of conductor $\varpi^{-\delta}$ and a multiplicative character $\eta : F^\times \to \C^\times$.

\begin{definition}\label{def:loca-shalika}
We say that an admissible  representation $\Pi$ of $G$ has a local $(\eta,\psi)$-Shalika model if there is a non-trivial (and hence injective) intertwining of $G=\GL_{2n}(F)$-modules
$$ \cS^{\eta}_{\psi}:\Pi \hookrightarrow \Ind_{S}^{G}(\eta\otimes\psi). $$
\end{definition}

For any $W\in  \Ind_{S}^{G}(\eta\otimes\psi)$ and for any quasi-character $\chi:F^\times \to \C^\times$  the zeta integral
\begin{equation}\label{eq:local-shalika-at-p}
\zeta(s;W,\chi)=\int_{\GL_n(F)} W\left(\left( \begin{array}{ccc}
h &  0\\
0 &  1_n
\end{array}\right)\right) \chi(\det(h))  |\det(h)|^{s-\frac{1}{2}} dh
\end{equation}
 is absolutely convergent for $\Re(s)\gg 0$. The following result is due to Friedberg and Jacquet.

\begin{prop}\cite[Prop. 3.1, 3.2]{friedjac}\label{prop:FJL-fct}
Assume that $\Pi$ has an $(\eta,\psi)$-Shalika model. Then for each $W\in \cS^{\eta}_{\psi}(\Pi)$ there is a holomorphic function $P(s;W,\chi)$ such that
$$\zeta(s;W,\chi)=L(s,\Pi\otimes \chi)P(s;W,\chi).$$
One may analytically continue $\zeta(s;W,\chi)$ by re-defining it as $L(s,\Pi\otimes \chi)P(s;W,\chi)$ for all $s\in\C$. Moreover, there exists a  vector $W_{\Pi}\in \cS^{\eta}_{\psi}(\Pi)$ such that  all unramified quasi-characters $\chi:F^\times\to\C^\times$ and  every $s\in\C$ one has
$$P(s;W_{\Pi},\chi)=(q^{s-\frac{1}{2}} \chi(\varpi))^{\delta n}. $$ 
If $\Pi$ is  spherical, then $W_{\Pi}$ can be taken to be the spherical vector $W_{\Pi}^\circ\in \cS^{\eta}_{\psi}(\Pi)$ normalized by the condition $W_{\Pi}^\circ({\bf 1}_{2n})=1$.
\end{prop}

For ramified twists we need the following refinement of  Proposition \ref{prop:FJL-fct}.

\begin{prop}\label{prop:localbirch}
Let $W\in \cS^{\eta}_{\psi}(\Pi)$ be a parahoric  invariant vector,  {\it i.e.}, 
\begin{equation}  \label{eq:shalikaproperty}
  W  \left(\begin{pmatrix}  h& \\   & h  \end{pmatrix}
  \begin{pmatrix}  {\bf 1}_n& X\\   & {\bf 1}_n  \end{pmatrix} \, g \, k  \right)=
  \eta(\det h)\psi(\tr X)W(g),
\end{equation}
for all $h\in\GL_n(F),$ $X\in M_n(F),$ $g\in G$ and $k\in J $.
Then for every finite order character $\chi:F^\times\to\C^\times$  of conductor  $\beta\geqslant 1$, and for  all $s\in\C$ with $\Re(s)\gg 0$   one has 
$$\zeta(s;W(-\cdot \xi t_\varpi^\beta),\chi) =   \cG(\chi)^n \cdot q^{\beta n\left(1-n\right)+(\beta+\delta) n\left(s-\frac{1}{2}\right)} 
W(t_\varpi^{-\delta}). $$
\end{prop}

\smallskip
\begin{proof}
For any $h\in\GL_{n}(F)$  and $ X\in M_n(\cO)$ the Shalika property \eqref{eq:shalikaproperty} implies that: 
\[
W\left(\begin{pmatrix}h&\\&{\bf1}\end{pmatrix}\xi t_\varpi^\beta\right)
= 
W\left(\begin{pmatrix}h&\\&{\bf1}\end{pmatrix}\xi t_\varpi^\beta\begin{pmatrix}{\bf1}&X\\&{\bf1}\end{pmatrix}\right) 
=
\psi(\tr (h \varpi^\beta  X w_n))\cdot
W\left( \begin{pmatrix}h&\\&{\bf1}\end{pmatrix} \xi t_\varpi^\beta\right),
\]
hence the zeta integral is supported over $\GL_n(F)\cap \varpi^{-\beta-\delta}M_n(\cO)$.  Also for $h\in\GL_{n}(F)$: 
\[
W\left(\begin{pmatrix}h&\\&{\bf1}\end{pmatrix}\xi t_\varpi^\beta\right)
 =  W\left(\begin{pmatrix}{\bf1}_n& h\\&{\bf1}_n\end{pmatrix}
\begin{pmatrix}h& \\&w_n\end{pmatrix}
 t_\varpi^\beta\right)
=  \psi(\tr h)\cdot W\left(\begin{pmatrix}h \varpi^\beta&\\&{\bf1}_n \end{pmatrix} \right).
\]
Using this and changing the variable $h\mapsto h \varpi^{-\beta'}$ with $\beta'=\beta+\delta$ yields
\begin{multline}\label{eq:zeta2}
\zeta(s; W(-\cdot \xi t_\varpi^\beta),\chi) \ = \\
\int_{\GL_n(F)\cap M_n(\cO)}
W\left(\left(\begin{smallmatrix}h\varpi^{-\delta}&\\&{\bf1}_n\end{smallmatrix}\right)\right)
 \psi(\tr (h\varpi^{-\beta'}))  (\chi|\cdot|^{s-\frac{1}{2}})(\det (h\varpi^{-\beta'})) dh.
\end{multline}

Denote by $(e_{ij})_{1\leqslant i,j\leqslant n}$ the standard basis of $M_n(\cO)$. 
Since $W$ is parahoric invariant, 
for any $i\neq j$ and $c\in \cO$, right translation by ${\bf1}_{n}+c e_{ij}\in \SL_n(\cO)$ in \eqref{eq:zeta2} yields: 
\begin{multline*}
\zeta(s; W(-\cdot \xi t_\varpi^\beta),\chi) \ = \\
\int dh \,W\left(\left(\begin{smallmatrix}h\varpi^{-\delta}&\\&{\bf1}_n\end{smallmatrix}\right)\right)
 \psi(\tr (h\varpi^{-\beta'}))  (\chi|\cdot|^{s-\frac{1}{2}})(\det (h\varpi^{-\beta'})) 
 \int_{\cO} dc\,\psi(c h_{ji}\varpi^{-\beta'}), 
\end{multline*}
and observe that $\int_{\cO} \psi(c h_{ji}\varpi^{-\beta'})dc=0$ unless $h_{ji}\in \cP^{\beta}$. 

Similarly right translation by ${\bf1}_{n}+ (c-1)e_{ii}$ with  $c\in \cO^\times$ shows that \eqref{eq:zeta2} equals: 
\[\int W\left(\left(\begin{smallmatrix}h&\\&{\bf1}_n\end{smallmatrix}\right)\right)
 \psi((\tr (h)-h_{ii})\varpi^{-\beta'}))  (\chi|\cdot|^{s-\frac{1}{2}})(\det (h\varpi^{-\beta'})) 
 \left(\int_{\cO^\times}
 \psi(c h_{ii} \varpi^{-\beta'})  \chi(c)d^\times c \right) dh,\]
  and $\left(\int_{\cO^\times}  \psi(c h_{ii} \varpi^{-\beta'})  \chi(c)d^\times c \right)=0$ unless $h_{ii}\in \cO^\times$
as  $\beta\geqslant 1$ equals  the conductor of $\chi$.

Therefore one can further restrict the domain of integration in \eqref{eq:zeta2}  to the  congruence subgroup 
$\ker\left(\GL_n(\cO)\to \GL_n(\cO/\cP^{\beta})\right)\cdot T_n(\cO),$ which by the Iwahori decomposition,  
may be identified to the product $N^{-}_n(\cP^{\beta})\times T_n(\cO)\times N_n(\cP^{\beta})$, where  $T_n$ denotes the diagonal subgroup of $\GL_n$ and
 $N_n$ denotes the  unipotent radical of the  standard Borel subgroup $B_n$.  Hence 
$$
\zeta(s; W(-\cdot \xi t_\varpi^\beta),\chi) \ = \ 
q^{\beta' n\left(s-\frac{1}{2}\right)}  W(t_\varpi^{-\delta})
\int_{N^{-}_n(\cP^{\beta})T_n(\cO)N_n(\cP^{\beta})}
 \psi(\tr (k\varpi^{-\beta'}))  \chi(\det (k\varpi^{-\beta'}))dk 
 $$
which can be simplified as 
\begin{multline*}
q^{\beta n\left(1-n\right)+\beta' n\left(s-\frac{1}{2}\right)}   W(t_\varpi^{-\delta})\prod_{1\leqslant i\leqslant n}
\int_{\cO^\times}
 \psi(t_i \varpi^{-\beta'})  \chi(t_i \varpi^{-\beta'})d^\times t_i \ = \\ 
 q^{\beta n\left(1-n\right)+(\beta+\delta) n\left(s-\frac{1}{2}\right)} 
 W(t_\varpi^{-\delta})\cdot \cG(\chi)^n, 
\end{multline*}
  as desired.  \end{proof}

\subsection{Non-vanishing of a local twisted zeta integral} \label{sec:nonvanising}
In order to ensure the non-vanishing of the  local twisted  Shalika integral in Proposition \ref{prop:localbirch}, which is
crucial for our applications,  one has to exhibit a  parahoric-spherical Shalika function $W$ on $G$ such  that $W(t_\varpi^{-\delta}) \neq 0$.
Assume that $\Pi$ is a spherical representation isomorphic to $\Ind_{B}^{G}(|\cdot|^{\frac{2n-1}{2}} \lambda)$ as in \eqref{eq:spherical} and let 
$\alpha_{i}=\lambda_i(\varpi)$, $1\leqslant i\leqslant 2n$.  Consider an unramified character $\eta$ of $F^\times$.

\begin{definition}\label{Q-regular}
Let $\tau\in W_G/W_H$ thought of as an $n$-element subset of $\{1,\dots, 2n\}$ (see \eqref{eq:tau}). 
We say that $\widetilde\Pi=(\Pi, \tau)$ is {\it $Q$-regular}  if it satisfies the following two conditions:
\begin{enumerate}
\item $q^{\frac{n(1-n)}{2}}  \prod_{i\in \tau} \alpha_{i}$ is a simple
eigenvalue for  $U_\cP=[J t_\varpi J ]$ acting on $ \Pi^{J }$, 
\smallskip
\item there exists $\rho\in \mathfrak{S}_{2n}$ such that for all $i\in \tau,$ $\rho(i)\notin \tau$  and 
$\alpha_{i} \alpha_{\rho(i)} = q^{2n-1} \eta(\varpi).$   
\end{enumerate}
\end{definition}

 Assume that  $\widetilde\Pi=(\Pi, \tau)$ is $Q$-regular.  Then  (i) together with Lemma \ref{lemma-Jacquet}  implies
\begin{equation}
\label{eqn:Q-reg}
 \prod\limits_{i\in \tau, j\notin \tau} ( \alpha_{i} - \alpha_{j})\neq 0, 
 \end{equation}
 while (ii) implies  by \cite[Prop. 1.3]{ash-ginzburg} that $\Pi$   admits  a $(\eta,\psi)$-Shalika model.

Without loss of generality assume from now on that $\tau=\{n+1,\dots,2n\}$ and that 
$\rho\in \mathfrak{S}_{2n}$  is the order $2$ element such that $\rho(i)=n+i$ for all $1\leqslant i\leqslant n$. 
In     \cite[(1.3)]{ash-ginzburg} the authors construct an $(\eta, \psi)$-Shalika functional on $\Pi$ 
sending $f\in  \Ind_{B_{2n}}^{\GL_{2n}}(|\cdot|^{\frac{2n-1}{2}} \lambda)$ to 
\begin{equation} \label{eqn:shalika-integral}
\cS(f)(g)  =   \int\limits_{B_n \backslash \GL_n} \int\limits_{M_n} 
f\left(\left( \begin{smallmatrix}  & {\bf 1}_n \\ {\bf 1}_n & X \end{smallmatrix} \right)
\left( \begin{smallmatrix} h &  \\ & h  \end{smallmatrix} \right) g \right)
\eta^{-1}(\det(h)) \bar\psi({\rm tr}(X)) \, dX\, dh
\end{equation}

By \cite[Lem. 1.5]{ash-ginzburg}, this integral converges in a certain domain and, when multiplied by \eqref{eqn:Q-reg},
can be  analytically  continued to $\C^{2n}$,  thus makes sense whenever \eqref{eqn:Q-reg} is non-zero. 
Let $f_0 \in\Ind_{B}^{G}(|\cdot|^{\frac{2n-1}{2}} \lambda)$ be the unique parahoric-spherical  function 
supported  on $B w_{2n} J $ and characterized by $f_0\left(\left( \begin{smallmatrix}  {\bf 1}_n  &\\ & \varpi^{-\delta}{\bf 1}_n \end{smallmatrix} \right)w_{2n}\right) = q^{-\delta n^2} $.  The following analogue of \cite[Lem. 1.4]{ash-ginzburg} holds.

\begin{lemma} \label{lem:shalika-eigenvector}
Let $W=\cS(f_0)$.  Then  $W(t_\varpi^{-\delta}) =1$. Moreover
$U_\cP\cdot f_0= q^{\frac{n(1-n)}{2}}(\prod_{i=n+1}^{2n} \alpha_i) f_0$. 
\end{lemma}

\begin{proof}
By the Iwasawa decomposition  $\GL_n = B_n K_n$ and as $\iota(K_n,K_n) \subset J $  we see that 
\begin{multline*}
W(t_\varpi^{-\delta})= \cS(f_0)(t_\varpi^{-\delta}) = 
\int_{M_n} 
f_0\left(\left( \begin{smallmatrix}  & {\bf 1}_n \\ {\bf 1}_n & X \end{smallmatrix} \right) t_\varpi^{-\delta} \right)
\bar\psi({\rm tr}(X)) \, dX \ = \\
\ = \ \int_{M_n} 
f_0\left(
\left( \begin{smallmatrix}  {\bf 1}_n  &\\ & \varpi^{-\delta}{\bf 1}_n \end{smallmatrix} \right)
\left( \begin{smallmatrix}  & {\bf 1}_n \\ {\bf 1}_n & \varpi^{\delta}X \end{smallmatrix} \right) \right)
\bar\psi({\rm tr}(X)) \, dX \ = \\
\ = \ f_0\left(\left( \begin{smallmatrix}  {\bf 1}_n  &\\ & \varpi^{-\delta}{\bf 1}_n \end{smallmatrix} \right)w_{2n}\right) q^{\delta n^2} 
\int_{M_n} 
f_0\left(
\left( \begin{smallmatrix}  & {\bf 1}_n \\ {\bf 1}_n & X \end{smallmatrix} \right) \right)
\bar\psi({\rm tr}(\varpi^{-\delta}X)) \, dX=1. 
\end{multline*}
One checks that  $
\left( \begin{smallmatrix}  & {\bf 1}_n \\ {\bf 1}_n & X \end{smallmatrix} \right)
 \in  Bw_{2n} J $ if and only if $X \in M_n(\cO)$,  in which case $\psi({\rm tr}(\varpi^{-\delta}X))=1$.
 The parahoric  decomposition     of $J=(J\cap U^-)(J \cap Q)=(J \cap U^-)(J \cap Q)$
implies  
\begin{equation}
J t_\varpi J=\bigsqcup_{m\in M_n(\cO/\cP)} \left(\begin{smallmatrix} {\bf1}_n &m \\ &{\bf1}_n \end{smallmatrix} \right)t_\varpi  J.  
\end{equation}
By  \eqref{eq:bruhat} it suffices to compute  $(U_\cP\cdot f_0)(\rho)$ for all $\rho\in W_G$.  By the above decomposition 
\[(U_\cP\cdot f_0)(\rho)=\sum_{m\in M_n(\cO)/M_n(\cP)} f_0\left(\rho \left(\begin{smallmatrix} \varpi{\bf1}_n &m \\ &{\bf1}_n \end{smallmatrix} \right)\right). \]
Note that  $\rho \left(\begin{smallmatrix} {\bf1}_n &m \\ &{\bf1}_n \end{smallmatrix} \right)t_\varpi $ belongs to the support 
 $B w_{2n} J = B w_{2n} t_\varpi J = B w_{2n}  J^- t_\varpi  $ of $f_0$ if and only if 
$\rho \left(\begin{smallmatrix} {\bf1}_n &m \\ &{\bf1}_n \end{smallmatrix} \right)\in K\cap  B w_{2n}  J^- =
(K\cap  B) w_{2n}  J^-=w_{2n}  J^- $ (see  \eqref{eq:bruhat}) which implies $\rho=w_{2n}$ and $m\in M_n(\cP)$. 
Hence $(U_\cP\cdot f_0)(\rho)=0$ for all  $\rho\neq w_{2n}$, while  
 $(U_\cP\cdot f_0)(w_{2n})=f_0\left(w_{2n} \left(\begin{smallmatrix} \varpi{\bf1}_n & \\ &{\bf1}_n \end{smallmatrix} \right)\right)
 =f_0\left(\left(\begin{smallmatrix} {\bf1}_n & \\ &\varpi {\bf1}_n \end{smallmatrix} \right)w_{2n} \right)=
  q^{\frac{n(1-n)}{2}}(\prod_{i=n+1}^{2n} \alpha_i)  f_0(w_{2n}) $.
\end{proof}

\section{\bf $L$-functions for $\GL_{2n}$}

\subsection{Global Shalika models and periods}
\label{sec:shalika}

This subsection contains a brief review of the necessary ingredients from \cite{GR-ajm} and a discussion involving $p$-adically integrally refined Betti--Shalika periods. 
Henceforth, $\Pi$ will stand for a (not necessarily unitary) cuspidal automorphic representation of 
$G(\A) = \GL_{2n}(\A_F)$.  
Keeping multiplicity one for $\GL_{2n}$ in mind, we will let $\Pi$ also stand for its representation space  within the space of cusp forms for $G(\A)$. 
Fix the non-trivial  additive unitary character  
$\psi: \A_{F}/F\longrightarrow \A/\Q \longrightarrow \C^\times$
where the first map is  the trace, whereas the second is the
usual additive character $\psi_0$ on $\A/\Q$ 
characterized by $\ker(\psi_0|_{\Q_\ell})=\Z_\ell$ for every prime
number $\ell$ and $\psi_0|_{\R}(x)=\exp(2\pi ix)$. 
We remark that $(\varpi_{v}^{-\delta_v})$, where $\delta_v $ is 
the valuation at $v$ of the different
$\mathfrak{d}$ of $F$, is the largest ideal
contained in $\ker(\psi_v)$. The discriminant  of $F$ is $\N_{F/\Q}(\mathfrak{d})$.

\subsubsection{\bf Global Shalika models}
\label{sect:globalShalikamodels}
 Let $\eta : F^\times \backslash 
\A_F^\times \to \C^\times$ be a Hecke character
such that $\eta^n$ equals the   central character $\omega_\Pi$ of $\Pi$. We get an automorphic character:
$$\eta \otimes \psi: S(F) \backslash S(\A_F)\to \C^\times, \quad \left( \begin{array}{ccc}
h &  hX\\
0 &  h
\end{array}\right) \mapsto  \eta(\det(h))\psi(Tr(X)).$$

For a cusp form $\varphi\in\Pi$ and $g\in G(\A)$ consider the integral
\begin{equation}\label{eq:shalika-functional}
W^{\eta}_{\varphi}(g)=\int_{Z(\A)S(F)\backslash S(\A_F)} \varphi(sg)(\eta \otimes \psi)^{-1}(s) ds,
\end{equation}
where   Haar measures  are normalized as in  \cite[\S2.8]{GR-ajm}.
It is well-defined by the cuspidality of the function $\varphi$ (see \cite[\S8.1]{jacshal}) and hence yields a function
$W^{\eta}_{\varphi}: G(\A)\rightarrow\C$ such that 
$$
W^{\eta}_{\varphi}(sg) = (\eta\otimes\psi)(s)\cdot W^{\eta}_{\varphi}(g),
$$
for all $g\in G(\A)$ and $s\in S(\A)$. In particular, we obtain an intertwining of $G(\A)$-modules
\begin{equation}\label{eq:shalika-intertwining}
\cS^\eta_\psi: \Pi\rightarrow \textrm{Ind}_{S(\A)}^{G(\A)}(\eta\otimes\psi), \quad \varphi\mapsto W^{\eta}_{\varphi}. 
\end{equation}
The following theorem, due to Jacquet and Shalika, gives a necessary and sufficient conditions for the existence of 
a non-zero intertwining as in \eqref{eq:shalika-intertwining}.

\begin{theorem}{\rm(\cite[Thm. 1]{jacshal})}
\label{thm:JS}
The following assertions are equivalent:
\begin{enumerate}
\item[(i)] There exists $\varphi\in\Pi$  such that $W^{\eta}_{\varphi} \neq 0$.
\item[(ii)] There exists an  injection of $G(\A)$-modules 
$\Pi\hookrightarrow \textrm{\emph{Ind}}_{S(\A)}^{G(\A)}(\eta\otimes\psi)$.
\item[(iii)] 
The twisted partial exterior square $L$-function $\prod_{v\notin \Sigma_\Pi} L(s, \Pi_v, \wedge^2 \otimes\eta^{-1}_v)$
has a pole at $s=1$, where $\Sigma_\Pi$ is the set of places where  $\Pi$ is ramified. 
\end{enumerate}
\end{theorem}

This is proved in \cite{jacshal} for unitary representations and its extension to the non-unitary case is easy.
If $\Pi$ satisfies any one, and hence all, of the equivalent conditions of Theorem~\ref{thm:JS}, then we say that $\Pi$  has an {\it $(\eta,\psi)$-Shalika model},  and
we call the isomorphic image $\cS^\eta_\psi(\Pi)$ of $\Pi$ under \eqref{eq:shalika-intertwining} a {\it global $(\eta,\psi)$-Shalika model} of $\Pi$.  Then clearly  $\Pi\otimes \chi$  has an $(\eta\chi^2,\psi)$-Shalika model for any  Hecke character $\chi$, by keeping the same model and only twisting the action.

The following proposition gives another equivalent condition for $\Pi$ to have a global Shalika model. 

\begin{prop}[Asgari--Shahidi \cite{asgari-shahidi}] 
\label{prop:selfdual}
Let $\Pi$ be a cuspidal automorphic representation of $\GL_{2n}(\A_F)$ with central character $\omega_\Pi$. Then the following assertions are equivalent:
\begin{enumerate}
\item[(i)] $\Pi$ has a global $(\eta,\psi)$-Shalika model for some character $\eta$ satisfying $\eta^n=\omega_\Pi$.
\item[(ii)] $\Pi$ is the transfer of a globally generic cuspidal automorphic representation $\pi$ of 
${\rm GSpin}_{2n+1}(\A_F)$.
\end{enumerate}
In particular, if any of the above equivalent conditions is satisfied, then $\Pi$ is essentially self-dual, {\it i.e.,} $\Pi\cong \Pi^\vee\otimes \eta$. The character $\eta$ may be taken to be the central character  of $\pi$.
\end{prop}

\subsubsection{\bf Period integrals and $L$-functions}
The following proposition,  due to Friedberg and Jacquet, is crucial for much that will follow. It relates the period-integral over $H$ of a cusp form $\varphi$ of $G$ to a certain zeta integral of the
function $W^{\eta}_{\varphi}$ in the Shalika model corresponding to $\varphi$ over one copy of $\GL_n$.

\begin{prop}\cite[Prop. 2.3]{friedjac} \label{prop:FJ}
Assume that  $\Pi$ has an $(\eta,\psi)$-Shalika model. 
For   $\varphi\in\Pi$ 
$$
\Psi(s,\varphi,\chi,\eta)=\int_{Z(\A)H(\Q)\backslash H(\A)} \varphi\left(\left( \begin{array}{ccc}
h_1 &  0\\
0 &  h_2
\end{array}\right)\right)(\chi | \cdot|^{s-\frac{1}{2}})\left(\frac{\det(h_1)}{\det(h_2)}\right)\eta^{-1}(\det(h_2)) dh_1 dh_2
$$
 converges absolutely for all $s\in\C$. For  $\Re(s)\gg 0$ it is equal to 
$$
\zeta(s;W^{\eta}_{\varphi},\chi)=\int_{\GL_n(\A_F)} W^{\eta}_{\varphi}\left(\left( \begin{array}{ccc}h &  0\\0 &  1 \end{array}\right)\right) \chi(\det(h)) |\det(h)|^{s-\frac{1}{2}} \, dh.
$$
thus providing  an analytic continuation of $\zeta(s;W^{\eta}_{\varphi},\chi)$ to all of  $\C$.
\end{prop}

Suppose the representation $\Pi$ of  $G(\A) = \GL_{2n}(\A_F)$ decomposes as $\Pi = \otimes'_v \Pi_v$, where $\Pi_v$ is an irreducible admissible representation of $\GL_{2n}(F_v)$. 

If $\Pi$ has a global Shalika model, then $\cS^\eta_\psi$ defines local Shalika models at every place (see Definition \ref{def:loca-shalika}). The corresponding local intertwining operators are denoted by $\cS^{\eta_v}_{\psi_v}$ and their images by $\cS^{\eta_v}_{\psi_v}(\Pi_v)$, whence $\cS^\eta_\psi(\Pi)=\otimes'_v \cS^{\eta_v}_{\psi_v}(\Pi_v)$.
We can now consider cusp forms $\varphi$ such that the function $W_{\varphi}\in \cS^\eta_\psi(\Pi)$ is factorizable as $W_\varphi=\otimes'_v W_{\varphi_v},$
where
$$
W_{\varphi_v}\in \cS^{\eta_v}_{\psi_v}(\Pi_v)\subset \textrm{Ind}_{S(F_v)}^{\GL_{2n}(F_v)}(\eta_v\otimes\psi_v).
$$
Then the  following  factorisation  holds for $\Re(s) \gg 0$:
\begin{equation}\label{eq:product-zeta}
 \zeta(s;W_{\varphi},\chi)=\prod_{v } \zeta_v(s;W_{\varphi_v},\chi_v), 
 \end{equation}
where the non-Archimedean local zeta integrals $\zeta_v(s;W_{\varphi_v},\chi_v)$ are
related  to $L$-functions in Proposition \ref{prop:FJL-fct}.

Proposition \ref{prop:FJ} relates this  Shalika zeta integral to a period integral over $H$, and 
 the main thrust of \cite{ash-ginzburg}, refined and generalized in \cite{GR-ajm}, is that 
the period integral over $H$ admits a cohomological interpretation, provided that $\Pi$ is of cohomological type.

\subsubsection{\bf Shalika models and  cuspidal cohomology}
\label{sect:cuspidal-coh}

In this paragraph we recall some well-known facts from Clozel~\cite[\S3]{Clo} (see also \cite[\S3.4]{GR-ajm}). Assume from now on that the cuspidal automorphic representation $\Pi$ is cohomological
with respect to a  dominant integral weight $\mu\in X_+^*(T)$ (see \eqref{eq:pure}), {\it i.e.},   
$$
\rH^q(\lieg_\infty,K_\infty^\circ ;\Pi\otimes V^\mu_{\C}) = 
\rH^q(\lieg_\infty,K_\infty^\circ ;\Pi_\infty\otimes V^\mu_{\C})\otimes\Pi_f \ \neq \ 0
$$
for some degree $q$. A necessary condition for the non-vanishing of this cohomology group is that the weight $\mu$ is pure, {\it i.e.}, 
$\mu\in X_0^*(T)$.  For each archimedean place $\sigma\in \Sigma_{\infty}$,  $\Pi_\sigma$ can be explicitly described as follows. 
 For any integer $\ell \geqslant 1$  consider  the unitary discrete series representation  $D(\ell)$ of $\GL_2(\R)$ of lowest non-negative 
${\rm SO}_2$-type $\ell+1$ and central character  $\mathrm{sgn}^{\ell+1}$.
Let $P$ be the parabolic subgroup of $\GL_{2n}$ with Levi factor $\prod_{i=1}^n \GL_2$.  Then 
$$
\Pi_\sigma\simeq\textrm{Ind}^{\GL_{2n}(\R)}_{P(\R)}
\left(\bigotimes_{i=1}^n D(2(\mu_{\sigma,i}+n-i)+ 1-\w)\otimes |\det|^{-\w/2}\right),
$$
in particular $\omega_{\Pi_\sigma}=|\cdot|^{-n \w}$.
The highest degree supporting cuspidal cohomology of $G$ is $t=|\Sigma_\infty|(n^2+n-1).$
For any character $\epsilon$ of $K_{\infty}/K_\infty^\circ $ 
the $\epsilon$-eigenspace of 
\begin{equation}\label{eq:gK-sigma}
  \rH^{t}(\lieg_\infty,K_\infty^\circ  ; \Pi_\infty \otimes V^{\mu}_{\C}) \ = \ 
  \Hom_{K_\infty^\circ }
  \left(\wedge^{t}(\lieg_\infty/\liek_\infty), \Pi_\infty \otimes V^{\mu}_{\C}\right)
  \end{equation}
is a line.   If in addition  $\Pi$  admits  an $(\eta,\psi)$-Shalika model then  $\eta$ is forced to be algebraic of  
the form $\eta = \eta_0 | \cdot |_F^{-\w}$ with $\eta_0$ of finite order, $\w$ is the purity weight of $\mu$ (see \cite[Thm.  5.3]{gan-raghuram}). 
Using multiplicity one theorem for local  Shalika models \cite{nien}  (see also \cite{chen-sun}), one deduces that for
any character $\epsilon$ of $K_{\infty}/K_\infty^\circ$ :
\[ \rH^{t}(\lieg_\infty,K_{\infty}^\circ ;\cS^{\eta_\infty}_{\psi_\infty}(\Pi_\infty)\otimes V^\mu_{\C})[\epsilon] \] 
is a line, a basis $\Xi_\infty^\epsilon$ of which we fix in way 
compatible  with twisting (see \cite[Lem. 5.1.1]{GR-ajm}).
The  relative Lie algebra cohomology of $\Pi$ as above  is a summand of the cuspidal cohomology which in turn injects into the cohomology with compact supports
(see \cite[\S2]{gan-raghuram})
\begin{equation}
  \label{eqn:gK-cohomology-to-compact-supports}
  \rH^{t}(\lieg_\infty,K_{\infty}^\circ ;\Pi^K\otimes V^\mu_{\C}) \ \hookrightarrow \ 
  \rH^{t}_{\rm cusp}(\SGK, \cV_{\C}^\mu)  \ \hookrightarrow \ 
  \rH^{t}_c(\SGK, \cV_{\C}^\mu). 
\end{equation}

  We  define an isomorphism
$\Theta^\epsilon$ of $G(\A_f)$-modules as the composition 
\begin{align}\label{eqn:theta-epsilon}
\cS^{\eta_f}_{\psi_f}(\Pi_f)  \xrightarrow{\sim} &
\cS^{\eta_f}_{\psi_f}(\Pi_f) \otimes
\rH^{t}(\lieg_\infty,K_{\infty}^\circ ; \cS^{\eta_{\infty}}_{\psi_{\infty}}(\Pi_\infty) \otimes V^\mu_{\C})[\epsilon]   \xrightarrow{\sim}  \nonumber \\
\xrightarrow{\sim} &\rH^{t}(\lieg_\infty,K_{\infty}^\circ ; \cS^{\eta}_{\psi}(\Pi) \otimes V^\mu_{\C})[\epsilon] \xrightarrow{\sim} 
\rH^{t}(\lieg_\infty,K_{\infty}^\circ ; \Pi \otimes V^\mu_{\C})[\epsilon],
\end{align}
where the first map is $W_f \mapsto W_f \otimes \Xi_\infty^\epsilon$, the 
 second map is the natural one and the third map is the map induced in cohomology
by $(\cS^\eta_\psi)^{-1}$ from \eqref{eq:shalika-intertwining}. Taking $K$-invariants in  \eqref{eqn:theta-epsilon} and composing with 
 \eqref{eqn:gK-cohomology-to-compact-supports} yields a Hecke equivariant embedding 
 \begin{align}\label{eqn:theta-epsilon-K}
\Theta^\epsilon_K: \cS^{\eta_f}_{\psi_f}(\Pi_f)^K \hookrightarrow  \rH^{t}_c(\SGK, \cV_{\C}^\mu)[\epsilon]. 
\end{align}

The reader should appreciate that  the analytic condition on $\Pi$ of admitting a Shalika model and  the algebraic condition of contributing to the cuspidal cohomology of $G$  are of  entirely different nature.  One may construct examples of representations satisfying only one of these conditions and not the other (see \cite[\S 3.5]{GR-ajm}).

\subsection{Ordinarity and  regularity}
\label{sec:local-aspects-ordinary-zeta}

For $\gp$ dividing $p$, we let $\varpi_\gp$ denote an uniformizer of $F_\gp$ and let 
 $q_\gp=\left|\varpi_\gp \right|_p^{-1}$ denote the cardinality of its residue field. 
Let  $\Sigma_\infty=\coprod_{\gp\mid p}  \Sigma_\gp$ be the partition induced by $i_p:\bar\Q\hookrightarrow \bar\Q_p$, where
 $\Sigma_\gp=\{\sigma:F_\gp\hookrightarrow \bar\Q_p\}$.
 
Recall from \S\ref{sec:weights-repns} that a weight $\mu \in X^*_+(T) $ yields a rational character $\mu:T=\Res_{F/\Q}T_{2n}\to\GL_1$, therefore induces a character
  $\mu_p=\otimes_{\gp\mid p}\mu_\gp$ of $T(\Q_p) = \prod_{\gp \mid p} T_{2n}(F_\gp)$,  
  where 
  \begin{equation}
  \begin{split}  
      &\mu_\gp:T_{2n}(F_\gp)\to \bar\Q_p^\times \text{ is given by }  (\mu_{\sigma})_{\sigma\in \Sigma_\gp}
  \text{  subject to the dominance condition }  \\
   &\mu_{\sigma,1}\geqslant  \mu_{\sigma,2}\geqslant \dots \geqslant \mu_{\sigma,2n},  \text{ for all }   \sigma\in \Sigma_\gp. 
   \end{split}
  \end{equation}

Recall the maximal $(n,n)$-parabolic  subgroup $Q\subseteq  G$. 
Given a cuspidal automorphic representation $\Pi$ of $G(\A)$ that is cohomological with respect to the  weight  $\mu
 \in X^*_+(T)$, we say that $\Pi_{\gp}$ is {\em $Q$-ordinary} (resp. {\em $B$-ordinary}) as  in \cite{Hida1995,Hida1998}.  
 
 Assume from now on that $\Pi_\gp$ is unramified  for all $\gp\mid p$. Since  $\Pi$ is cohomological
    there exists an
unramified {\em algebraic} character $\lambda_{\gp}:T_{2n}(F_\gp)\to\bar\Q^\times\subset \C^\times$ such that  (see 
  \eqref{eq:spherical}):  
 \begin{equation}
   \Pi_\gp=  \Ind_{B_{2n}(F_\gp)}^{\GL_{2n}(F_\gp)}(|\cdot|^{\frac{2n-1}{2}}\lambda_{\gp}).
  \end{equation}

Using $i_p: \C\xrightarrow{\sim} \bar\Q_p$ allows us to see the Hecke parameters  $\alpha_{\gp,i}=\lambda_{\gp,i}(\varpi_\gp), 1\leqslant i\leqslant 2n$ as elements of $\bar\Q_p^\times$. Then  $\Pi_\gp$ is  $B$-ordinary  relative to an ordering of its Hecke parameters $\alpha_{\gp,i}$  if and only if
  \begin{equation}
  \label{eqn:ordinary-individual-terms}
     \left|   \mu_{\gp,i}^\vee(\varpi_\gp) \cdot q_\gp^{1-i} \alpha_{\gp, 2n+1-i} \right|_p =1, \quad \mbox{for all $1\leqslant i \leqslant 2n$.}
  \end{equation}
 where  $\left|\cdot \right|_p$  denotes the  $p$-adic norm.  
  The $B$-dominance condition \eqref{eq:B-dominant} then implies that: 
  \begin{equation}  \label{eqn:ordering}
 \left| \alpha_{\gp, 1}  \right|_p< \left| \alpha_{\gp, 2}  \right|_p< \dots < \left| \alpha_{\gp, 2n}\right|_p, 
  \end{equation}
 hence there exists at most one ordering of the Hecke parameters for which  $\Pi_\gp$ is  $B$-ordinary. 
 Moreover this implies that a $B$-ordinary $\Pi_\gp$  is necessarily regular, {\it i.e.}, the $\alpha_{\gp,i}$ are pairwise distinct.

  Similarly, $\Pi_\gp$ is   $Q$-ordinary relative to  $\tau\in W_G/W_H$  if and only if
   \begin{equation}  \label{eqn:ordinary}
    \prod_{i\in \tau}\left|\alpha_{\gp,i}\right|_p =
    \left|q_\gp^{\frac{n(n-1)}{2}} \mu_\gp^\vee(\iota(\varpi_\gp^{-1}{\bf1}_{n},{\bf1}_{n}))\right|_p .
  \end{equation}
We will make a key observation that   $Q$-ordinarity implies  $Q$-regularity (see Definition \ref{Q-regular}).

\begin{lemma}\label{lem:ord-reg}
Assume  that the cuspidal automorphic representation $\Pi$ of $G(\A)$ is 
cohomological with respect to $\mu$ and admits an $(\eta,\psi)$-Shalika model. 
For $\gp$ dividing  $p$, if $\Pi_{\gp}$  is spherical and $Q$-ordinary, 
 then  
 \[  v_p( \alpha_{\gp,i}) < |\Sigma_\gp|\tfrac{\w+2n-1}{2} < v_p( \alpha_{\gp,i'}), \text { for all } i\in \tau,  i'\notin \tau.\] 
 In particular, $\Pi_{\gp}$ is  $Q$-ordinary only relative to $\tau$. Moreover,  
 $\widetilde\Pi_{\gp}=(\Pi_{\gp},\tau)$ is $Q$-regular, {\it i.e.},   $q_\gp^{\frac{n(1-n)}{2}}  \prod_{i\in \tau} \alpha_{\gp,i}$ is a simple eigenvalue of $U_\gp$ acting on  $ \Pi_\gp^{J_\gp}$ and  
 $\lambda_{\gp, i}\neq \lambda_{\gp, i'}$
for all $i\in \tau,  i'\notin \tau$.  
\end{lemma}

\begin{proof}  
Consider  the Hecke operators  $U_{\gp,n-1}=[I_{\gp}t_{\gp,n-1} I_{\gp}]$ acting on $\Pi_{\gp,N}^{I_\gp}$, where $t_{\gp,n-1}=\diag(\varpi_\gp{\bf1}_{n-1},{\bf1}_{n+1})\in \GL_{2n}(F_\gp)$. Since  $\mu_\gp^\vee(t_{\gp,n-1})\cdot U_{\gp,n-1}$ preserves $p$-integrality its eigenvalues on  $\Pi_{\gp,N}^{I_\gp}$ are  $p$-integral, in particular for any $i\in \tau$ we have 
 \begin{equation}
   \prod_{i\neq i' \in \tau}  \left| \alpha_{\gp,i'} \right|_p \leqslant  \left|\mu_\gp^\vee(t_{\gp,n-1}^{-1})\cdot q_\gp^{\frac{(n-1)(n-2)}{2}} \right|_p.
  \end{equation}
Together with  \eqref{eqn:ordinary} this implies that   $\left|\alpha_{\gp,i} \right|_p \geqslant 
\left|\mu_{\gp,n}^\vee(\varpi_{\gp}^{-1}) 
q_\gp^{n-1}\right|_p=\left|\mu_{\gp,n+1}(\varpi_{\gp}) 
q_\gp^{n-1}\right|_p $, {\it i.e.}, 
  \begin{equation}\label{eq:Q-ordinary}
    v_p( \alpha_{\gp,i}) \leqslant \sum_{\sigma\in  \Sigma_\gp} \left(n-1+ \mu_{\sigma,n+1}\right).
  \end{equation}
 The existence of  $(\eta_\gp,\psi_\gp)$-Shalika model for   $\Pi_\gp$ gives by 
  \cite[Prop. 1.3]{ash-ginzburg}  an $i'=\rho(i)$  so that 
 \begin{equation}
v_p(\alpha_{\gp,i})+v_p(\alpha_{\gp,i'}) = |\Sigma_\gp|(2n-1+\w).
  \end{equation}
   The latter equality together  with  \eqref{eq:pure} and \eqref{eq:Q-ordinary}    yields for all $ i\in \tau$: 
    \[
       v_p( \alpha_{\gp,i}) \leqslant \sum_{\sigma\in  \Sigma_\gp} \left(n-1+ \mu_{\sigma,n+1}\right)< |\Sigma_\gp|\tfrac{\w+2n-1}{2} <
        \sum_{\sigma\in  \Sigma_\gp} \left(n+ \mu_{\sigma,n}\right)   \leqslant v_p( \alpha_{\gp,\rho(i)}).
  \]
All claims follow then easily  as clearly $\rho(\tau)\cap\tau=\varnothing$ as required by 
  Definition \ref{Q-regular}.
  \end{proof}

\subsection{$p$-adic interpolation of critical values}\label{sec:p-adic-interpol}

  We suppose in the sequel that  $\Pi$ is a  cuspidal automorphic representation of $G(\A)$ 
  which is cohomological with respect to a pure weight  $\mu
 \in X^*_0(T)$ and  that $\Pi$ admits an $(\eta,\psi)$-Shalika model. 
Assume further that for all $\gp\mid p$, $\Pi_\gp$ is spherical and  that  $\widetilde\Pi_\gp=(\Pi_\gp, \tau)$ is {\it $Q$-regular} for
$\tau=\{n+1,\dots, 2n\}$ in the sense of Definition \ref{Q-regular}.

\subsubsection{\bf Choice of local Shalika vectors}
\label{sec:localshalikachoice}
For $v\nmid p\infty$ we recall the  vector $W_{\Pi_v}\in \cS^{\eta_v}_{\psi_v}(\Pi_v)$  from  Proposition \ref{prop:FJL-fct}. 
  
  For  $\gp \mid p$, $\Pi_\gp$ is spherical and $\widetilde\Pi_\gp=(\Pi_\gp, \{n+1,\dots, 2n\})$ is {\it $Q$-regular}, 
  in particular,    $\alpha_\gp=\prod_{n+1\leqslant i \leqslant 2n}\alpha_{\gp,i}$
  is a simple eigenvalue for the Hecke operator $U_\gp$ acting on $\Pi_\gp^{J_\gp}$. By Lemma  \ref{lem:shalika-eigenvector} there exists an unique  $W_{\widetilde\Pi_\gp}$ on the line $ \cS^{\eta_\gp}_{\psi_\gp}(\Pi_\gp)^{J_\gp}\left[U_\gp-\alpha_\gp\right]$ normalized so that  $W_{\widetilde\Pi_\gp}(t_\gp^{-\delta_\gp})=1$.  Let 
\begin{equation}\label{eq:finite-representative}  
 W_{\widetilde\Pi_f} =   \otimes_{\gp\mid p}W_{\widetilde\Pi_\gp} \bigotimes \otimes_{v\nmid p\infty}'W_{\Pi_v}  
   \in\cS^{\eta_f}_{\psi_f}(\Pi_f).
   \end{equation}
   
In addition to the conditions  \ref{condition-K1} and \ref{condition-K2} on $K$ (see \S\ref{sec:distributions}) henceforth we assume that: 
\begin{enumerate}[label=(K3), ref=(K3)]
\item \label{condition-K3} $K$ fixes $W_{\widetilde\Pi_f}$ and $\eta$ is trivial on $I(\gm)$ hence can be seen as a character of $\Cl(\gm)$.
\end{enumerate}

For the local vectors at infinity, given any character $\epsilon$ of $K_{\infty}/K_\infty^\circ $
we recall the  basis    $\Xi_\infty^\epsilon$ of the line 
    $\rH^{t}(\lieg_\infty,K_{\infty}^\circ ;\cS^{\eta_\infty}_{\psi_\infty}(\Pi_\infty)\otimes V^\mu_{\C})[\epsilon]$ from \S\ref{sect:cuspidal-coh}. As in \cite[\S 4.1]{GR-ajm}, we make the following 

\begin{definition}\label{def:global-test-vector} Fix a basis $\{e_\alpha \}$ of $V^\mu_{\C}$ and a 
basis $\{\omega_i\}$ of $\left(\lieg_\infty/\liek_\infty\right)^\vee$. For  $\underline i= (i_1,...,i_{t}),$ with  $1 \leqslant i_1 < \cdots < i_{t} \leqslant \dim\left(\lieg_\infty/\liek_\infty\right)$, we
let $\omega_{\underline i}= \omega_{i_1}\wedge...\wedge \omega_{i_{t}}\in \bigwedge^{t} \left(\lieg_\infty/\liek_\infty\right)^\vee.$

      \smallskip
Using  \eqref{eq:gK-sigma},  $\Xi_\infty^\epsilon$ 
   can be written as a $K_\infty^\circ$-invariant element 
  \begin{equation}\label{eq:infinityrepresentative}
   \Xi_\infty^\epsilon \ = \ 
   \sum_{\underline i,\alpha}
   \omega_{\underline i}\otimes W_{\infty,\underline i, \alpha}^\epsilon \otimes e_\alpha 
   \ \in \ 
   \wedge^{t}(\lieg_\infty/\liek_\infty)^\vee \otimes \cS^{\eta_\infty}_{\psi_\infty}(\Pi_\infty)\otimes V^{\mu}_{\C}.
  \end{equation}
  for a unique choice of  $W_{\infty, \underline i, \alpha}^\epsilon \in \cS^{\eta_\infty}_{\psi_\infty}(\Pi_\infty)$
  called `cohomological vectors at infinity'.  
  
For $\underline i$ and $\alpha$ as above,  let $\varphi_{\underline i, \alpha}^\epsilon \in \Pi$ be the unique vector 
      whose image under  \eqref{eq:shalika-intertwining} equals:
   $$    
   W^{\eta}_{\varphi_{\underline i, \alpha}^\epsilon}=W_{\widetilde\Pi_f} \otimes W_{\infty,\underline i, \alpha}^\epsilon \in\cS^{\eta}_{\psi}(\Pi).   
   $$
    \end{definition}

    For each character $\epsilon$ of $K_{\infty}/K_{\infty}^\circ$ the embedding $\Theta^\epsilon_K$ defined in \eqref{eqn:theta-epsilon-K} yields 
 \[\cS^{\eta_f}_{\psi_f}(\Pi_f) \xrightarrow{\Theta^\epsilon_K}       \rH^{t}_c(\SGK, \cV_{\C}^\mu)[\epsilon] 
 \xrightarrow{i_p^*}     \rH^{t}_c(\SGK, \cV_{\bar\Q_p}^\mu)[\epsilon]
 \xrightarrow{(g,v)\mapsto (g, g^{-1}\cdot v)}  \rH^{t}_c(\SGK, \cV_{\bar\Q_p}^\mu)[\epsilon].        \]
 The image $\phi_{\widetilde\Pi}^\epsilon$ of $W_{\widetilde\Pi_f}$ under the composition of these three maps belongs to $\rH^{t}_c(\SGK, \cV_{E}^\mu)[\epsilon]$
 for some finite extension $E$ of $\Q_p$, and after possibly rescaling the maps $\Theta^\epsilon$, {\it i.e.,} rescaling the basis elements $\Xi_\infty^\epsilon$, one can  render the
 cohomology class $\cO$-integral:
   \begin{equation}\label{eq:coh-class}
    \phi_{\widetilde\Pi}^\epsilon\in \rH^{t}_c(\SGK, \cV_{\cO}^\mu)[\epsilon].  
 \end{equation}
 Recall that the Hecke operator $U_{p^\beta}^\circ$  defines an endomorphism of $\rH^{t}_c(\SGK, \cV_{\cO}^\mu)[\epsilon]$.
 Since $W_{\widetilde\Pi_f}$ is an $U_{p^\beta}$-eigenvector with eigenvalue $\alpha_{p^\beta}=\prod_{\gp\mid p}\alpha_\gp^{\beta_\gp}$, it follows (after possibly additionally rescaling by a power of  $p$ killing the torsion in $\rH^{t}_c(\SGK, \cV_{\cO}^\mu)$) that one can assume $\phi_{\widetilde\Pi}^\epsilon$ is an $U_{p^\beta}^\circ$-eigenvector  with eigenvalue
 $\alpha_{p^\beta}^\circ=\mu^\vee(t_p^\beta) \alpha_{p^\beta}$.

\subsubsection{\bf  Interpolation formula at critical points.}
In this section we will relate the image of $\phi_{\widetilde\Pi}^\epsilon$ defined in \eqref{eq:coh-class} by the evaluation map 
$\cE^{j,\eta}_{\beta,[\delta]}$ from  \eqref{eq:eval}, to the  Friedberg-Jacquet integral from Proposition \ref{prop:FJ}.

\begin{prop} \label{prop:integration-one-component}
For any character $\epsilon$ of $K_{\infty}/K_\infty^\circ $ and any  $[\delta]\in\Cl(p^\beta\gm)\times \Cl(\gm)$ we have: 
 \[
 i_p^{-1}\left( \mu^\vee(t_p^{-\beta})\cdot  \cE^{j,\w}_{\beta,[\delta]}(\phi_{\widetilde\Pi}^\epsilon)\right)= \int_{\SHbeta[\delta]} 
\varphi_{\widetilde\Pi,j}^\epsilon(h  \xi t_p^\beta) |\det(h_1^j h_2^{\w-j}) |_F dh, 
\]
where $\varphi_{\widetilde\Pi,j}^\epsilon = \sum_{\underline i} \sum_\alpha a_{\underline i, \alpha, j}^\epsilon \cdot  \varphi_{\underline i, \alpha}^\epsilon$ 
for suitable  $a_{\underline i, \alpha, j}^\epsilon \in \C$. 
\end{prop}

In the above proposition, and henceforth,  
$\sum_{\underline i}$ will denote summing over all  ${\underline i} = (i_1,\dots,i_{t})$ and $\sum_\alpha$ will denote summing over all 
$1 \leqslant \alpha \leqslant \dim(V^\mu),$ as in Definition \ref{def:global-test-vector}. A more careful choice of the bases $\{e_\alpha\}$ and $\{\omega_i\}$ as in \cite[\S7]{J-preprint} yields algebraic coefficients $a_{\underline i,\alpha,j}^\epsilon\in\overline{\Q}$.

\begin{proof} We follow closely the proof of \cite[Prop. 4.1]{BDJ}. Consider the    commutative diagram:
	$$	\xymatrix{
		\rH^{t}_c(\SGK, \cV_{E}^\mu) \ar[d]^{\kappa_j \circ \mathcal{T}_{\beta}}\ar[rrrrr]_{\sim}^{(g,v)\mapsto (g, g^{-1}\cdot v)} &&&&& 
		\rH^{t}_c(\SGK, \cV_{E}^\mu)  \ar[d]^{\kappa_j \circ \tau_\beta\circ \iota_\beta^*}\\
		\rH^{t}_c(\SHbeta, \cV^{(j,\w-j)}_{E}) \ar[rrrrr]_{\sim}^{(h,v)\mapsto (h, h^{-1}\cdot v)} 
		\ar[d]^{(\relbar\cap \theta_{[\delta]})\circ\triv'_\delta{}^*}&&&&& 
		\rH^{t}_c(\SHbeta, \cV^{(j,\w-j)}_{E}) \ar[d]^{(\relbar\cap \theta_{[\delta]})\circ\triv_\delta^*}\\
		E \ar[rrrrr]_{\sim}^{ \N_{F_p/\Q_p}\left(\det(\delta_{1,p}^j\delta_{2,p}^{\w-j})\right) }&&&&&E}
		$$	
		where $ \tau_\beta=\mu^\vee(t_p^{-\beta})\tau_\beta^\circ$ is defined in \eqref{eqn:tau-beta-on-cohomology}, 
		the horizontal maps are induced from the morphisms of local systems written above them, 
	the map $\mathcal{T}_{\beta}$ is induced from the morphisms of local systems $(h,v)\mapsto (h\xi t_p^\beta, v)$, and  $\triv'_\delta$ 
	is induced from the morphisms of local systems:
	$$H(\Q)\delta L_\beta H_\infty^\circ \times V^{(j,\w-j)}_{\Q(\widetilde\Pi)} \to \cV^{(j,\w-j)}_{\Q(\widetilde\Pi)| \SHbeta[\delta]} \text{ , } 
	(\gamma\delta\ell h_\infty,v)\mapsto(\gamma\delta\ell h_\infty, \gamma^{-1}\cdot v).$$

Since $\cE^{j,\w}_{\beta, [\delta]}= \varepsilon \!\left(\det(\delta_{1}^j\delta_{2}^{\w-j})\right) \cE^{j,\w}_{\beta,\delta}$, with  
$\varepsilon_f = |\cdot |_{F,f}\N_{F_p/\Q_p} $, the above diagram shows that the proposition  is equivalent to: 
 \begin{equation}\label{eq:archimedean-int} 
 |\det(\delta_{1,f}^j\delta_{2,f}^{\w-j})|_F 
\left(\relbar\cap \theta_{[\delta]})\circ\triv_\delta^*\circ\kappa_j\circ \mathcal{T}_{\beta}\right)(\phi_{\widetilde\Pi}^\epsilon)
=  \int_{\SHbeta[\delta]} 
\varphi_{\widetilde\Pi,j}^\epsilon(h\xi t_p^\beta) |\det(h_1^j h_2^{\w-j}) |_F dh, 
\end{equation}
the left hand side being  considered over $\C$ via $i_p^{-1}:\bar\Q_p \xrightarrow{\sim} \C$. By Definition  \ref{def:global-test-vector} 
$$ 
 \phi_{\widetilde\Pi}^\epsilon=   (\cS^\eta_\psi)^{-1}( W_{\widetilde\Pi_f} \otimes \Xi_\infty^\epsilon) \ = \ 
\sum_{\underline i} \sum_{\alpha} 
   \omega_{\underline i}\otimes \varphi_{\underline i, \alpha}^\epsilon \otimes e_\alpha 
\ \in \ 
 \left(\wedge^{t}(\lieg_\infty/\liek_\infty)^\vee\otimes 
\Pi \otimes V^{\mu}_{\C}\right)^{K_\infty^\circ},
$$
yielding   a  $V^{\mu}_{\C}$-valued differential $t$-form on $G_\infty^\circ/K_\infty^\circ$. Now, recall the basis  $\kappa_j$ of the line  $\Hom_H(V^{\mu}, V^{(j,\w-j)})$ from \eqref{eqn:hom-H-one-dim} 
and  consider the  map
$$
   \kappa_j\circ \iota^*: \rH^{t}(\lieg_\infty,K_\infty^\circ ;\Pi_\infty \otimes V^{\mu}_{\C})\to \rH^{t}(\lieh_\infty,L^\circ_\infty;\Pi_\infty \otimes V^{(j,\w-j)}_{\C}). 
 $$
We get 
$$
(\kappa_j\circ \mathcal{T}_{\beta})(\phi_{\widetilde\Pi}^\epsilon) \ = \  
 \sum_{\underline i} \sum_\alpha
 \iota^* \omega_{\underline i} \otimes \varphi_{\underline i, \alpha}^\epsilon(-\cdot \xi t_p^\beta) \otimes \kappa_j(e_\alpha) \ \in \ 
 \left(\wedge^{t}(\lieh_\infty/\liel_\infty)^\vee\otimes 
\Pi \otimes V^{(j,\w-j)}_{\C}\right)^{L_\infty^\circ}.
$$ 

 First, let   $\kappa_j^\circ: V^{(j,\w-j)}_{\C}\xrightarrow{\sim} \C$ be the scalar extension of \eqref{eq:basis-j-line}, and so $\kappa_j(e_\alpha)$ corresponds to a complex number. Next, 
 after the discussion in the paragraph following \eqref{eqn:q0}, we can fix a  basis  for the top-exterior 
 $\wedge^{t}(\lieh_\infty/\liel_\infty)^\vee$ corresponding to the Haar measure $dh_\infty$; hence $ \iota^* \omega_{\underline i}$ is a scalar multiple of $dh_\infty$. 
 Putting both together, the restriction to $\SHbeta[\delta]$  of $\kappa_j( \mathcal{T}_{\beta}( \phi_{\widetilde\Pi}^\epsilon))$
 can be seen as $V^{(j,\w-j)}_{\C}$-valued top-degree differential form on $H_\infty^\circ/L_\infty^\circ$ given by 
$$ 
\sum_{\underline i} \sum_\alpha  
a_{\underline i, \alpha, j}^\epsilon \cdot  
\varphi_{\underline i, \alpha}^\epsilon(\delta h_\infty \xi t_p^\beta) \det(h_{1,\infty}^{ j} h_{2,\infty}^{ \w-j})
dh_\infty
$$
for suitable  $a_{\underline i, \alpha, j}^\epsilon\in \C$. Writing $h=\gamma \delta l h'_\infty\in H(\Q)\delta L_\beta H_\infty^\circ\subset H(\A)$, 
and using that 	$\triv'_\delta(\gamma\delta\ell h'_\infty,v)=(\gamma\delta\ell h'_\infty, \det(\gamma_{1,\infty}^{-j}\gamma_{2,\infty}^{j-\w}) v)$ 
one obtains \eqref{eq:archimedean-int} and the Proposition from 
\begin{multline*}
|\det(\delta_{1,f}^j\delta_{2,f}^{\w-j})|_F \det(\gamma_{1,\infty}^{-j}\gamma_{2,\infty}^{j-\w})\det(h_{1,\infty}^{ j} h_{2,\infty}^{ \w-j}) \ = \\
|\det(\delta_{1,f}^j\delta_{2,f}^{\w-j})|_F\det(h_{1,\infty}^{\prime j} h_{2,\infty}^{\prime \w-j}) \ = \ |\det(h_1^j h_2^{\w-j})|_F.\qedhere
\end{multline*} 
 \end{proof}

\subsubsection{\bf $p$-adic distributions attached to $\widetilde\Pi$}
Recall that  $\Pi$ is a cuspidal automorphic representation of $G(\A)$ admitting a global $(\psi,\eta)$-Shalika model, 
which is cohomological with respect to a pure dominant integral weight $\mu$. Recall also that   $\Pi_\gp$ is spherical for all  $\gp\mid p$
and  that  $\widetilde\Pi_\gp=(\Pi_\gp, \{n+1,\dots, 2n\})$ is {\it $Q$-regular} (see Definition \ref{Q-regular}), which by Lemma \ref{lem:ord-reg} is automatically 
fulfilled if  $\Pi_\gp$ is   $Q$-ordinary. In all cases 
$\Pi_\gp^{J_\gp}$ contains a unique line on which $U_\gp$ acts by   $\alpha_{\gp}$. 
Finally recall the $U_\gp^\circ$-eigenvectors  $\phi_{\widetilde\Pi}^\epsilon$ constructed in \eqref{eq:coh-class}.  Then 
\begin{equation}\label{eq:def-total-phi}
\phi_{\widetilde\Pi}=\sum_{\epsilon: F_{\infty}^\times/F_{\infty}^{\times\circ}  \to \{\pm 1\}}\phi_{\widetilde\Pi}^\epsilon
  \end{equation}
is  an $U_\gp^\circ$-eigenvector with same eigenvalue. 
When  $\widetilde\Pi_p$  is $Q$-ordinary, consider the element 
\begin{equation}
    \mmu_{\widetilde\Pi}^\eta=\mmu_{\phi_{\widetilde\Pi}}^\eta=\varepsilon_{\cyc}^{-j}(\mmu^{j,\eta}_{\phi_{\widetilde\Pi}})\in \cO[[\Cl(p^\infty)]],
  \end{equation}      
     constructed in \eqref{eq:definitionofmu-fs} and \eqref{eq:def-mu-pi} which defines a 
      measure $d\mmu_{\widetilde\Pi}^\eta$ on $\Cl(p^\infty)$.      
 
\subsubsection{\bf Main theorem on $p$-adic interpolation} Fix any character $\epsilon$ of $K_{\infty}/K_\infty^\circ $ and any $j\in \Crit(\mu)$. Consider the following cohomological test vector: 
\begin{equation}
\label{eqn:coh-test-vector-infty}
W_{\Pi_\infty,j}^\epsilon \ =  \ \sum_{\underline i,\alpha} a_{\underline i, \alpha, j}^\epsilon W_{\infty, \underline i, \alpha}^\epsilon.
\end{equation}
A crucial result of Sun \cite[Thm. 5.5]{sun} asserts the following non-vanishing:
\begin{equation}
\label{eqn:sun}
\zeta_\infty(j+\tfrac{1}{2};W_{\Pi_\infty,j}^\epsilon) \in \C^\times.
\end{equation}
Note that since $\Pi_\infty \otimes {\rm sgn} = \Pi_\infty$ we have suppressed  $\chi_\infty$ from the notation.

In fact, using K\"unneth's theorem its easy to see that $\zeta_\infty(j+\tfrac{1}{2};W_{\Pi_\infty,j}^\epsilon)$ is a product of similar quantities parsed over the archimedean places.
Since we have multiplicity one for cuspidal cohomology in top-degree (see \eqref{eq:gK-sigma}) we can only change the class $\Xi_\infty^\epsilon$ by a non-zero scalar, which will correspondingly 
scale the cohomological test-vector and so also the zeta integral. 
The variation of the complex period $\zeta_\infty(j+\tfrac{1}{2};W_{\Pi_\infty,j}^\epsilon)$ in $j$ is studied in \cite{J-preprint}. 
The reader is also referred to the discussion around \cite[Thm. 6.6.2]{GR-ajm}.

Recall the auxiliary ideal $\gm$ from \ref{condition-L1} in \S\ref{sec:shl} and, for brevity, let's define: 
\begin{equation}\label{eq:gamma}
\gamma=  \#\Cl(\gm)\cdot \# \GL_n(\cO_F/\gm)\cdot \# \PGL_n(\cO_F/\gm)\cdot \prod_{\gp\mid p}
\left(q_\gp^{-n^2} \cdot  \#  \GL_n(\cO_F/\gp)\right)  \in\Q^\times.
\end{equation}

\begin{theorem}
\label{thm:p-adic-interpolation} 
Let $\Pi$ be a cuspidal automorphic representation of $\GL_{2n}/F$ admitting a $(\psi,\eta)$-Shalika model and such that 
$\Pi_\infty$ is cohomological of weight $\mu$. Assume that for all  $\gp\mid p$, $\Pi_\gp$ is spherical and  admits a $Q$-regular refinement $\widetilde \Pi_{\gp}$. Then for any $j\in \Crit(\mu)$ and for any finite order  character $\chi$ of $\Cl(p^\infty)$ of  conductor $\beta_\gp \geqslant 1$ at  $\gp\mid p$:
\begin{multline*}
  i_p^{-1}\left(  \int_{\Cl(p^\infty)} \chi(x)  d\mmu_{\phi_{\widetilde\Pi}}^{\eta,j}(x)  \right)
    \ = \\ 
 =\gamma \cdot \N^{jn}_{F/\Q}(\mathfrak{d})  \cdot 
        \prod_{\gp\mid p}\left(\alpha_{\gp}^{-1} 
    q_\gp^{n(j+1)}\right)^{\beta_\gp} \cdot     
    {\cG(\chi_f)^{n} \cdot  L(j+\tfrac{1}{2}, \Pi_f \otimes \chi_f) } 
     \zeta_\infty(j+\tfrac{1}{2};W_{\Pi_\infty,j}^{(\varepsilon^j\chi\eta)_\infty}). 
\end{multline*}
\end{theorem}

\begin{proof}
Using  \eqref{eq:definitionofmu-fs} and  \eqref{eqn:summing-2nd-var} we find that $ \int_{\Cl(p^\infty)} \chi(x)  d\mmu_{\phi_{\widetilde\Pi}}^{\eta,j}(x)$ equals
  \[
   (\alpha_{p^{\beta}}^\circ)^{-1}
    \sum_{[x]\in \Cl(p^\beta\gm)}\chi([x])\cE^{j,\eta}_{\beta,[x]}(\phi_{\widetilde\Pi}) \nonumber
     =  \alpha_{p^{\beta}}^{-1} \cdot \mu^\vee(t_p^{-\beta}) \sum_{\substack{[x]\in \Cl(p^\beta\gm) \\ [y]\in \Cl(\gm)}}
\chi([x]) \eta_0([y]) \cE^{j,\w}_{\beta,[\delta(x,y)]}(\phi_{\widetilde\Pi}). 
  \]
Since $\pi_0(\SHbeta)\simeq \Cl(p^\beta \gm)\times  \Cl(\gm)$ by
 \eqref{eq:def-total-phi} and Proposition \ref{prop:integration-one-component}  the integral equals: 
 \begin{align*}
 \alpha_{p^{\beta}}^{-1} \cdot \sum_{\epsilon: \{\pm 1\}^{\Sigma_\infty}\to \{\pm 1\}} 
\int_{\SHbeta} \varphi_{\widetilde\Pi,j}^\epsilon(h\xi t_p^\beta) 
\chi\left(\frac{\det(h_1)}{\det(h_2)}\right)
\left| \frac{\det(h_1)}{\det(h_2)}\right|_F^j \eta^{-1}( \det(h_2)) dh.
  \end{align*}
  Note that the integrand is $L_\infty^\circ Z(\A_f)$-invariant and $L_\infty/L_\infty^\circ Z_\infty$ acts on it by 
  $\epsilon \varepsilon_\infty^j\chi_\infty\eta_\infty$, hence the integral vanishes unless $\epsilon=(\varepsilon^j\chi\eta)_\infty$.
  Since  $Z(\A_f)\cap L_\beta$ is independent of $\beta$, after some volume computation, one further finds:
 \begin{align*}
 i_p^{-1}\left(  \int_{\Cl(p^\infty)}\varepsilon^j(x) \chi(x)  d\mmu_{\widetilde\Pi}^\eta(x)\right) =\gamma \cdot  \Psi\left(j+\tfrac{1}{2},\varphi_{\widetilde\Pi,j}^{(\varepsilon^j\chi\eta)_\infty}(-\cdot \xi t_p^\beta),\chi,\eta\right)
   \prod_{\gp\mid p} (\alpha_{\gp}^{-1}  q_\gp^{n^2})^{\beta_\gp}. 
  \end{align*}

  By Proposition \ref{prop:FJ} the   Friedberg-Jacquet integral  has an Euler product for $\Re(s)\gg 0$:
  $$
{\Psi(s,\varphi_{\widetilde\Pi,j}^\epsilon(-\cdot \xi t_p^\beta),\chi,\eta)} =
  \prod_{v\nmid p\infty}\zeta_v(s;W_{\Pi_v},\chi_v)\cdot
  \prod_{\gp\mid p}\zeta_\gp(s;W_{\widetilde\Pi_\gp}(-\cdot\xi t_\gp^{\beta_\gp}),\chi_\gp)\cdot
  \zeta_\infty(s;W_{\Pi_\infty,j}^\epsilon, \chi_\infty). $$

   Since  $L(s, \Pi \otimes \chi)$ has trivial Euler factors at all places $\gp\mid p$ (as $\Pi_\gp$ is spherical while $\chi_\gp$ is ramified), Proposition \ref{prop:FJL-fct}  implies that: 
    $$\prod_{v\nmid p\infty}\zeta_v(j+\tfrac{1}{2};W_{\Pi_v},\chi_v)= \N_{F/\Q}^{jn}(\mathfrak{d}^{(p)}) \chi(\mathfrak{d}^{-1})^n L(j+\tfrac{1}{2}, \Pi_f \otimes \chi_f).$$
The  factor at $\gp\mid p$ is computed in Proposition \ref{prop:localbirch}, which together with   Lemma \ref{lem:shalika-eigenvector}  gives:
  \[
q_\gp^{\beta_\gp n^2}
\zeta_\gp(j+\tfrac{1}{2};W_{\widetilde\Pi_\gp},\chi_\gp)   =  
\cG(\chi_\gp)^n \cdot q_\gp^{\beta_\gp n+ (\beta_\gp+\delta_\gp)jn}W_{\widetilde\Pi_\gp}({\bf1}_{2n})=
\N_{F/\Q}^{jn}(\mathfrak{d}_p) \cG(\chi_\gp)^n \cdot q_\gp^{\beta_\gp n(j+1)}.
  \]
    Since $\cG(\chi_f)=\chi(\mathfrak{d}^{-1}) \prod_{\gp\mid p} \cG(\chi_\gp)$ and $\mathfrak{d}=\mathfrak{d}_p\mathfrak{d}^{(p)}$ we obtain the desired formula. 
    \end{proof}

\bigskip

\begin{proof}[Proof of  Theorem \ref{thm:padic-L}]
 By  Lemma \ref{lem:ord-reg}, for each $\gp\mid p$, the the  $Q$-ordinary refinement $\widetilde\Pi_\gp$ of $\Pi_\gp$ is $Q$-regular, hence 
 Theorem  \ref{thm:p-adic-interpolation}  applies.    The interpolation formula in Theorem \ref{thm:padic-L} then  follows immediately since by   
Theorem \ref{thm:manin}  one has
 \[\int_{\Cl(p^\infty)} \varepsilon^j(x) \chi(x)  d\mmu_{\widetilde\Pi}^\eta(x)=\int_{\Cl(p^\infty)}  \chi(x)  d\mmu_{\phi_{\widetilde\Pi}}^{\eta,j}(x).\qedhere\]  
    \end{proof}

\bigskip
\subsection{Non-vanishing of twists}
\subsubsection{\bf The main theorem.}

Having established the Manin relations in our context (see Theorem \ref{thm:manin} above), we can now prove a non-vanishing result for twisted $L$-functions using a 
method that goes back to Manin and Greenberg. Such a technique to prove non-vanishing of twists has also been used recently by Januszewski~\cite{J-preprint-2} for Rankin--Selberg $L$-functions and Eischen~\cite{ellen} for $L$-functions of unitary groups. However, our results, Theorem \ref{thm:non-vanishing} and 
Corollary \ref{cor:simutaenous} below, are not only independent of these other recent works but are also beyond the scope of results in both \cite{J-preprint-2} and \cite{ellen} as well as previous results obtained by analytic number theoretic methods (\cite{rohrlich}, \cite{barthel-ramakrishnan}, \cite{luo}, \cite{chinta-friedberg-hoffstein}). 

\medskip
\begin{theorem}
\label{thm:non-vanishing}
Let $\mu$ be a pure dominant integral weight such that 
\begin{equation}
\label{eqn:mu-regular-non-vanishing}
\mu_{\sigma,n} >  \mu_{\sigma, n+1}\text{, for all }  \sigma\in \Sigma_\infty. 
\end{equation}
Let $\Pi$ be a cuspidal automorphic representation which is cohomological with respect to the weight $\mu$ and admitting an  $(\eta,\psi)$-Shalika model. 
Assume that for all primes $\gp$ above a  prime number $p$, $\Pi_\gp$ is unramified and   $Q$-ordinary. 
Then for all $j\in \Crit(\mu)$ and for all but finitely many Dirichlet characters $\chi$ of $F$ of   $p$-power conductor we have: 
$$
L\left(\tfrac{1}{2}+j, \Pi \otimes (\chi\circ \N_{F/\Q}) \right) \neq 0.
$$
\end{theorem}

We begin with a few comments. Since 
$\Pi^\circ = \Pi \otimes |\cdot|^{\w/2}$ is a  unitary cuspidal automorphic representation, we see that 
$\tfrac{1+\w}{2}$ is the center of symmetry for the $L$-function of $\Pi$:  
$$
L\left(\tfrac{1+\w}{2}, \Pi \otimes \chi\right) = L\left(\tfrac{1}{2}, \Pi^\circ\otimes \chi\right). 
$$
By regularity one knows that $\Crit(\mu)$ is non-empty and 
condition  \eqref{eqn:mu-regular-non-vanishing} is equivalent to assuming that  $\Crit(\mu)$ has at least two elements. 
If $\Pi$ is unitary then $\w = 0$ and  $\tfrac{1}{2} \in \Crit(\Pi \otimes \chi) = \Crit(\mu)$, whence Theorem \ref{thm:main-theorem} is a particular case of  Theorem \ref{thm:non-vanishing}. 

\smallskip
If Leopoldt's conjecture holds for $F$ at $p$ then one readily obtains a statement for all but finitely many 
$p$-power conductor Hecke  characters, as opposed to Dirichlet characters. 

\smallskip

We show non-vanishing of critical values of twisted  $L$-functions by showing  non-vanishing statement about distributions
on the cyclotomic $\Z_p$-extension  of $F$.  Recall  the $p$-adic cyclotomic character 
\[\varepsilon :{\mathcal{C}\!\ell_{\Q}^+}(p^\infty)\xrightarrow{\sim}  \Z_p^\times =  \mu_{2p} \times (1+2p\Z_p)\]
 the first component of which is given by the Teichm\"uller character $\omega$, while the fixed field of the
kernel of the second component $\varepsilon \omega^{-1}$ is the cyclotomic $\Z_p$-extension of $\Q$. Then by a well-known result due to Serre there is an isomorphism $\cO[[1+2p\Z_p]]\simeq\cO[[T]]$  sending $1+2p$ to $1+T$. 
Composing with the norm map  $\N_{F/\Q}: \Cl(p^\infty)\to {\mathcal{C}\!\ell_{\Q}^+}(p^\infty)$  allows us to lift Dirichlet characters to Hecke characters over $F$, thus to push-forward of 
a measure   on $\Cl(p^\infty)$, such as $\mu_{\widetilde\Pi}$, to a measure on ${\mathcal{C}\!\ell_{\Q}^+}(p^\infty)$. 
Further composing with $\omega^m: \mu_{2p}\to \cO^\times$ for $0\leqslant m \leqslant p-1$   allows us to define a measure 
on $1+2p\Z_p$, {\it i.e.}, an element   $\omega^m(\mu_{\widetilde\Pi})\in \cO[[T]]$. 

\begin{proof}
We will first show that $\omega^m(\mu_{\widetilde\Pi}) \neq 0$  for all $m\in \Z$.
  By the interpolation property in Theorem \ref{thm:p-adic-interpolation}, the measure $\omega^m(\mu_{\widetilde\Pi})$ interpolates the 
  algebraic parts of  $L(\tfrac{1}{2} + j, \Pi \otimes \omega^{m-j}\chi)$ for $j\in\Crit(\mu)$ and $\chi$ runs over all  
  Dirichlet characters of    (non-trivial) $p$-power order and conductor. 
Our hypothesis \eqref{eqn:mu-regular-non-vanishing} implies that we find $j\in\Crit(\mu)$ satisfying $j> \tfrac{\w}{2}$, hence
 $\tfrac{1}{2} + j$ lies outside the interior of the critical strip $\tfrac{w}{2} < \Re(s) < \tfrac{w}{2}+1$
 for $L(s,\Pi)$ and thus  $L(\tfrac{1}{2} + j, \Pi \otimes \omega^{m-j}\chi)\neq 0$. 
 Therefore $\omega^m(\mu_{\widetilde\Pi}) \neq 0$ as claimed.

  By the Weierstrass preparation theorem, a non-zero element of $ \cO[[T]]$ admits only finitely many zeros in $\bar\Z_p$. 
  Again by Theorem \ref{thm:p-adic-interpolation} this means that, given {\it any} $j\in\Crit(\mu) $ and $m$,  there are at most finitely many Dirichlet characters $\chi$     of  $p$-power order and conductor such that $L(\tfrac{1}{2} + j, \Pi \otimes \omega^{m-j}\chi) = 0.$ Since any  
 $p$-power   conductor Dirichlet character is of that form for some $0\leqslant m \leqslant p-1$, the theorem follows. 
\end{proof}

\subsubsection{\bf Variations}

\begin{corollary}[Nearly-ordinary case]
\label{cor:stronger-non-vanishing}
Under the hypotheses of Theorem \ref{thm:non-vanishing}, let $\nu$ be a finite order  character of $\Cl(p^\infty)$. Then for all but finitely many Dirichlet  characters $\chi$ of finite order and with $p$-power conductor we have: 
$$ L(\tfrac{\w+1}{2}, \Pi\otimes\nu \chi)\neq 0.  $$
\end{corollary}

\begin{proof} Use the twisted norm map $[x]\mapsto \nu(x) [\N_{F/\Q}x]$ to push forward $\mu_{\widetilde\Pi}$ to a measure on ${\mathcal{C}\!\ell_{\Q}^+}(p^\infty)$.  Then proceed  {\it mutatis mutandis} as in the proof of  Theorem \ref{thm:non-vanishing}.
\end{proof}

This result is slightly stronger because the representation $\Pi \otimes \nu$, even though of cohomological type and admitting a 
Shalika model, is no longer ordinary at $p$, nor spherical.

The following corollary of Theorem~\ref{thm:non-vanishing} follows from the fact that we have non-vanishing 
{\it for all but finitely many} Dirichlet  characters $\chi$ of finite order and with $p$-power conductor. 

\begin{corollary}[Simultaneous non-vanishing] 
\label{cor:simutaenous}
 For $1 \leqslant k \leqslant r$ fix  $n_k\in \Z_{>0}$  and let 
$\mu_k$ be a pure  dominant integral weight for $\GL_{2n_k}$ over  $F$. Suppose that each $\mu_k$ satisfies the regularity condition in (\ref{eqn:mu-regular-non-vanishing}) and that its purity weight $\w_k$ is even. Let $\Pi_k$ be a cuspidal automorphic representation
of $\GL_{2n_k}(\A_F)$ of cohomological weight $\mu_k$ admitting a Shalika model. 
For a  prime number $p$, suppose that each $\Pi_k$ is unramified and  $Q$-ordinary  at $p$.  
Then, for all but finitely many Dirichlet  characters $\chi$ of   $p$-power conductor, we have: 
$$
L\left(\tfrac{\w_{1}+1}{2}, \Pi_1\otimes \chi\right)L\left(\tfrac{\w_2+1}{2}, \Pi_2 \otimes \chi\right) \cdots 
L\left(\tfrac{\w_r+1}{2}, \Pi_r \otimes \chi\right) \ \neq \ 0.
$$
\end{corollary}
Let's note that this is a simultaneous non-vanishing result at  the  central point. We will leave it to the reader to formulate the stronger version of simultaneous non-vanishing combining  Corollaries \ref{cor:stronger-non-vanishing} and 
 \ref{cor:simutaenous}.  

\medskip

As a very concrete example illustrating an application of simultaneous non-vanishing to algebraicity results, let's consider 
the  unitary cuspidal automorphic representation $\pi(\Delta)$ of $\GL_2(\A)$ associated to the 
 Ramanujan  $\Delta$-function. 
 A particular case of Corollary \ref{cor:simutaenous} gives infinitely many Dirichlet characters $\chi$ such that 
\[ 
L\left(17, \Sym^3(\Delta)\otimes \chi\right) L\left(6, \Delta\otimes \chi\right) \ = \ 
L\left(\tfrac{1}{2}, \Sym^3(\pi(\Delta))\otimes \chi\right) L\left(\tfrac{1}{2}, \pi(\Delta)\otimes \chi\right)  \neq  0. 
\]
For such a character we get from \cite[Cor. 5.2]{raghuram-imrn} the following identity of $L$-values: 
\[ 
L\left(\tfrac{1}{2}, \Sym^5(\pi(\Delta))\otimes \chi\right) \ = \ 
\frac{L\left(\tfrac{1}{2}, \Sym^3(\pi(\Delta)) \times \Sym^2(\pi(\Delta))\right)}
{L\left(\tfrac{1}{2}, \Sym^3(\pi(\Delta))\otimes \chi\right) L\left(\tfrac{1}{2}, \pi(\Delta)\otimes \chi\right)}. 
\]
Using the rationality result in \cite[Thm.  1.1]{raghuram-imrn} for the $L$-value in the numerator of the right hand side, and the 
 \cite[Thm.  1.3]{raghuram-imrn} for the $L$-values in the denominator of the right hand side,  
we get a new rationality result for: 
\[
L\left(\tfrac{1}{2}, \Sym^5(\pi(\Delta))\otimes \chi\right)  \ = \ 
L\left(28, \Sym^5(\Delta)\otimes \chi\right). 
\] 
Similarly, using the results of \cite{raghuram-forum}, and simultaneous non-vanishing for the central values of 
the first and third symmetric power $L$-functions of a Hilbert cusp form, one may now generalize this to get new rationality results 
for the symmetric fifth power $L$-functions of a Hilbert cusp form.

\begin{small}
\bibliographystyle{plain}

\end{small}

\end{document}